\documentclass[10pt,a4paper]{article}

\usepackage{amsmath,amssymb,amsthm}
\usepackage{enumerate}
\usepackage{xcolor}
\usepackage{url}
\usepackage{authblk}
\usepackage{tikz-cd}
\usepackage[normalem]{ulem}
\usepackage{algorithm}
\usepackage{algpseudocode}
\usepackage{booktabs}
\usepackage{multicol}
\usepackage{mathdots}
\usepackage{etoolbox}
\usepackage[T1]{fontenc}
\usepackage[colorlinks=true, linkcolor=blue, citecolor=blue, urlcolor=blue]{hyperref}

\patchcmd{\thebibliography}{\section}{\subsection}{}{}

\makeatletter
\g@addto@macro\@openbib@code{\setlength{\itemsep}{0pt}}
\makeatother

\newtheorem{theorem}{Theorem}[section]
\newtheorem{lem}[theorem]{Lemma}
\newtheorem{rmk}[theorem]{Remark}
\newtheorem{prop}[theorem]{Proposition}
\newtheorem{cor}[theorem]{Corollary}
\newtheorem{conj}[theorem]{Conjecture}

\definecolor{darkgreen}{rgb}{0,0.5,0}

\author[1, 4]{Luciano N. Grippo \thanks{lgrippo@campus.ungs.edu.ar}}
\author[2, 4]{Min Chih Lin \thanks{oscarlin@dc.uba.ar}}
\author[1, 4]{Verónica Moyano \thanks{vmoyano@campus.ungs.edu.ar}}
\author[3, 4]{Camilo Vera \thanks{cvera@ic.fcen.uba.ar}}
\affil[1]{Universidad Nacional de General Sarmiento. Instituto de Ciencias; Argentina.}
\affil[2]{Instituto de Cálculo and Departamento de Computación, Universidad de Buenos Aires, Argentina.}
\affil[3]{Instituto de Cálculo and Departamento de Matemática, Universidad de Buenos Aires, Argentina.}
\affil[4]{Consejo Nacional de Investigaciones Científicas y Técnicas, Argentina.}

\title{Counting perfect edge dominating sets: extremal results and linear-time algorithms}
\date{}

\begin{document}

\maketitle

\begin{abstract}
An edge of a graph \emph{dominates} itself and each edge adjacent to it. A \emph{perfect edge dominating set} is a subset of edges such that each edge outside the subset is dominated by exactly one edge of the subset. In this article, we characterize the extremal graphs on $n$ vertices in the classes of trees, forests, and chordal graphs with respect to the number of perfect edge dominating sets. Moreover, we derive linear-time algorithms for counting perfect edge dominating sets and for counting dominating induced matchings in generalized series-parallel graphs and chordal graphs.
\medskip

\end{abstract}

\section{Introduction}

The enumeration of combinatorial structures in graphs has been
extensively studied over the last decades. Problems such as counting
independent sets, matchings, and dominating sets play a central role in
combinatorics, graph algorithms, and complexity theory.

Counting problems often exhibit significantly higher computational
complexity than their associated decision problems. A foundational work
in the complexity theory of counting problems was published in 1979 \cite{LV-1979}. In particular, that work established the
\(\#P\)-completeness of several counting problems, including counting
perfect matchings, minimal vertex covers, maximal cliques, and directed
spanning trees.

A \emph{dominating set} of a graph \(G\) is a subset
\(S\subseteq V(G)\) such that every vertex outside \(S\) is adjacent to
at least one vertex of \(S\). The counting of dominating sets has
motivated a rich literature involving domination polynomials, recursive
formulas, and structural decompositions. The \emph{domination polynomial}
of a graph was introduced in \cite{AP-2014}. If \(G\) is a graph on
\(n\) vertices, its domination polynomial is
\[
D(G,x)=\sum_{i=\gamma(G)}^n d(G,i)x^i,
\]
where \(d(G,i)\) is the number of dominating sets of \(G\) of cardinality
\(i\), and \(\gamma(G)\) is the \emph{domination number} of \(G\), that
is, the minimum cardinality of a dominating set of \(G\).

In contrast, analogous counting questions for edge domination structures
remain far less explored. An edge \(e\) of a graph \(G\) is said to
\emph{dominate} itself and every edge adjacent to it.

A \emph{dominating induced matching} (DIM) of a graph \(G\), also called
an \emph{efficient edge dominating set} (EED), is a subset of edges such
that every edge of \(G\) is dominated by exactly one edge of the subset.
Not every graph admits a DIM. Moreover, all the DIMs of a graph, whenever
they exist, have the same cardinality \cite{Lu-Ko-Tang-2002}. The problem
of deciding whether a given graph admits a DIM is known as the
\emph{DIM problem}, and was introduced in
\cite{Gr-Sl-Sh-Ho-1993}. The DIM problem is NP-complete in general graphs
\cite{Gr-Sl-Sh-Ho-1993}, and remains NP-complete in \(p\)-regular graphs
for every \(p\geq 3\) \cite{CARDOSO20083060} and in planar bipartite
graphs \cite{Lu-Ko-Tang-2002}. However, it can be solved in linear time
in several graph classes, including \(P_7\)-free graphs
\cite{Br-Mo-2014}, claw-free graphs \cite{LinMS14}, generalized
series-parallel graphs, and chordal graphs \cite{Lu-Ko-Tang-2002}.

A \emph{perfect edge dominating set} (PED-set) of a graph is a subset of
edges such that every edge outside the subset is dominated by exactly one
edge within the subset. Unlike a DIM, every graph admits a PED-set,
namely its entire edge set. Furthermore, every DIM of a graph is a
PED-set of minimum cardinality \cite{Lu-Ko-Tang-2002}. The
\emph{PED problem}, introduced in \cite{Lu-Ko-Tang-2002}, takes as input
a graph \(G\) and an integer \(k\), and asks whether \(G\) has a PED-set
of cardinality at most \(k\). The PED problem is NP-complete in general
graphs and remains NP-complete in bipartite graphs
\cite{Lu-Ko-Tang-2002} and in \(H\)-free graphs whenever \(H\) is not a
linear forest \cite{Li-Lo-Mo-Sz}. On the positive side, it can be solved
in linear time in generalized series-parallel graphs and chordal graphs
\cite{Lu-Ko-Tang-2002}, and by a robust linear-time algorithm in
\(P_5\)-free graphs \cite{Li-Lo-Mo-Sz}.

Beyond finding a perfect edge dominating set of minimum cardinality, it
is natural to investigate the multiplicity of such structures. In a
recent work, we proved that deciding whether a DIM-less graph admits at
least two PED-sets is NP-complete
\cite{Grippo-Lin-Vera-2026}. Thus, even the most basic decision question
concerning the number of PED-sets is computationally difficult. This
provides further motivation for the systematic study of the number of
PED-sets admitted by a graph, from both algorithmic and extremal
perspectives.

The \emph{DIM-counting problem} for a graph \(G\) consists of determining
the number of dominating induced matchings admitted by \(G\). This number
is denoted by \(\mu(G)\) \cite{Moyano-2017-thesis}. The same thesis
establishes tight bounds for \(\mu(G)\) in the classes of general graphs,
triangle-free graphs, and connected graphs. Moreover, in \cite{LMS}, is presented
an \(O(n)\)-time algorithm for counting dominating induced matchings in
claw-free graphs.

The \emph{PED-counting problem} for a graph \(G\) consists of determining
the number of perfect edge dominating sets admitted by \(G\). Motivated
by the scarcity of results on the enumeration of edge domination
structures, we investigate extremal and algorithmic aspects of the
PED-counting problem. We identify the graphs on \(n\) vertices that
maximize the number of perfect edge dominating sets within the classes
of trees, forests, and chordal graphs. We also present linear-time
algorithms for counting perfect edge dominating sets and dominating
induced matchings in generalized series-parallel graphs and chordal
graphs.

The paper is organized as follows. Section~\ref{sec: preliminaries}
presents definitions and preliminary results used throughout the paper.
In Section~\ref{sec: extpro}, we study extremal questions concerning the
number of perfect edge dominating sets. We show that paths are extremal
among trees, determine the extremal forests, and characterize the
extremal chordal graphs. In Section~\ref{sec: alg}, we consider the
associated counting problems and, by adapting the structural algorithms
presented in \cite{Lu-Ko-Tang-2002}, derive linear-time algorithms
for counting perfect edge dominating sets and dominating induced
matchings in generalized series-parallel graphs and chordal graphs.
Finally, Section~\ref{sec: conc} contains concluding remarks and discusses
some open problems.

\section{Preliminaries}~\label{sec: preliminaries}

\subsection{Basic concepts}

Let $G$ be a graph. The vertex and edge sets of $G$ are denoted by $V(G)$ and $E(G)$, respectively. If $uv\in E(G)$, we say that $u$ (resp. $v$) is \emph{adjacent} to $v$ (resp. $u$); we also say that the edge $uv$ is \emph{incident} to $u$ and $v$. We say that a graph $H$ is a \emph{subgraph} of $G$ if $V(H)\subseteq V(G)$ and $E(H)\subseteq E(G)$; and it is an \emph{induced subgraph} if further $E(H)=\{uv\in E(G):u, v\in V(H)\}$. If $S\subseteq V(G)$, the \emph{subgraph induced by $S$} is denoted by $G[S]$, and $G-S$ stands for $G[V(G)\setminus S]$.

The \emph{disjoint union} of the graphs $G$ and $H$, denoted by $G + H$, is the graph whose vertex set is $V(G + H)=V(G)\cup V(H)$ and edge set is $E(G + H)=E(G)\cup E(H)$. If $k\in\mathbb{N}$, then $k\,G$ denotes the disjoint union of $k$ copies of $G$.

Two graphs $G$ and $H$ are \emph{isomorphic} if there exists a bijection $f:V(G)\longrightarrow V(H)$ such that $uv\in E(G)$ if and only if $f(u) f(v)\in E(H)$.

If $u$ is a vertex of a graph $G$, its \emph{neighborhood} in $G$ is $N_G(u):=\{v\in V(G):uv\in E(G)\}$, and its \emph{degree} in $G$ is $d_G(u):=|N_G(u)|$. When the context is clear, we write simply $N(u)$ and $d(u)$. A vertex $u\in V(G)$ is called a \emph{leaf} of $G$ if $d(u)=1$. For an induced subgraph $H$ of $G$ and a vertex $u\in V(G)$, define $N_H(u):=N_G(u)\cap V(H)$, and $d_H(u):=|N_H(u)|$. Two edges are \emph{adjacent} if they are incident at the same vertex.

Let $G$ be a graph, and let $S\subseteq V(G)$ and $v\not\in V(G)$ be a new vertex. The \emph{contraction of $S$ into $v$} consists of replacing the set $S$ with the single vertex $v$, and making every neighbor of $S$ adjacent to $v$. The resulting graph, denoted by $G_{S, v}$, has vertex set $V(G_{S, v})=(V(G)\setminus S)\cup\{v\}$ and edge set $E(G_{S, v})=(E(G)\setminus\{e\in E(G):e\,\,\text{is incident to a vertex of}\,\,S\})\cup\{vx:x\in N_G(S)\}$.

We write $P_k$ and $C_k$ to denote the \emph{path} and the \emph{cycle on $k$ vertices}, respectively. The cycle $C_3$ is called a \emph{triangle}. A graph $G$ is \emph{connected} if for all pairs of vertices of $G$ there exists a path that joins them. A \emph{tree} is a connected acyclic graph, and a \emph{forest} is an acyclic graph. A \emph{linear forest} is a forest whose connected components are paths. A graph is \emph{complete} if each pair of vertices of it is adjacent, and we write $K_n$ to denote the complete graph on $n$ vertices. A \emph{clique} is a set of pairwise adjacent vertices. An \emph{independent set} is a set of pairwise non-adjacent vertices.

A \emph{star} on $q+1$ vertices is a graph composed of $q$ leaves and one vertex of degree $q$, and it is denoted by $K_{1, q}$. The vertex of degree $q$ is called the \emph{center} of the star. A graph is called a \emph{chordal graph} if it does not contain any cycle on at least four vertices as an induced subgraph.

\subsection{Perfect edge dominating sets in general graphs}

Let $G$ be a graph. Given two edges $e$ and $f$ of $G$, we say that $e$ \emph{dominates} $f$ if $e=f$, or if they are adjacent. A subset $P\subseteq E(G)$ is a \emph{perfect edge dominating set} (PED-set) of $G$ if every edge in $E(G)\setminus P$ is dominated by exactly one edge of $P$. Note that a PED-set always exists in any graph, since the whole set of edges of a graph is a PED-set of it.

A \emph{coloring} of $G$ is an assignment of a label to each vertex of $G$. In our setting, the labels are the colors black ($\mathsf{b}$), yellow ($\mathsf{y}$), and white ($\mathsf{w}$). Let $P$ be a PED-set of $G$. We define the function $\sigma_P:V(G)\longrightarrow\{\mathsf{b}, \mathsf{y}, \mathsf{w}\}$ which partitions $V(G)$ into three sets $B$, $Y$ and $W$, where $v\in B$ (respectively $v\in Y$, $v\in W$) if and only if $\sigma_P(v)=\mathsf{b}$ (respectively $\sigma_P(v)=\mathsf{y}$, $\sigma_P(v)=\mathsf{w}$). These sets are defined as follows:
\begin{itemize}
     \item $B$ consists of the vertices incident with at least two edges of $P$,
     \item $Y$ consists of the vertices incident with exactly one edge of $P$,
     \item $W$ consists of the vertices incident with no edge of $P$.
\end{itemize}

We call $\sigma_P$ the \emph{3-coloring associated with $P$}. Such a coloring satisfies the following properties:
\begin{enumerate}
     \item $W$ is an independent set,
     \item A vertex $v\in V(G)\setminus W$ is colored yellow if and only if $v$ is a leaf of $G-W$,
     \item If $v$ is colored white, then all its neighbors are colored yellow, and
     \item If $v$ is colored black, then it has no neighbors colored white and $d(v)\geq 2$.
\end{enumerate}

Conversely, every coloring of $G$ satisfying the above conditions is associated with a perfect edge dominating set of $G$, namely the set of edges whose endpoints are assigned non-white colors.

Any coloring satisfying these conditions is called a \emph{valid 3-coloring of $G$}. Consequently, there is a bijection between the perfect edge dominating sets of $G$ and the valid 3-colorings of $G$.

There can be three possible types of PED-sets for a graph $G$. Let $P$ be a PED-set of $G$.
\begin{itemize}
     \item We say that $P$ is the \emph{trivial PED-set} if $P=E(G)$.
     \item We say that $P$ is a \emph{DIM} if no two edges in $P$ are adjacent.
     \item We say that $P$ is a \emph{proper PED-set} if it is neither the trivial PED-set nor a DIM.
\end{itemize}

Now, consider a connected graph $G$ and a PED-set $P$ of it. Let $(B, Y, W)$ be the partition of $V(G)$ induced by $P$.
\begin{itemize}
     \item If $|V(G)|=1$, then $W=V(G)$ and $B=Y=\emptyset$. In this case, $P$ is the trivial PED-set and a DIM.
     \item If $|V(G)|=2$, then $Y=V(G)$ and $B=W=\emptyset$. In this case, $P$ is the trivial PED-set and a DIM.
     \item If $|V(G)|\geq 3$, then:
     \begin{itemize}
          \item A 3-coloring induces the trivial PED-set if and only if $W=\emptyset$. Since $G$ is connected, there exists at least one vertex with degree at least two, and this vertex is assigned the black color.
          \item A 3-coloring induces a DIM if and only if $B=\emptyset$, $Y\neq\emptyset$ and $W\neq\emptyset$.
          \item A 3-coloring induces a proper PED-set if and only if $B\neq\emptyset, Y\neq\emptyset$ and $W\neq\emptyset$.
     \end{itemize}
\end{itemize}

The number of perfect edge dominating sets of a graph $G$ will be denoted by $\tilde{\mu}(G)$. Since every graph $G$ admits the trivial PED-set, and every DIM is also a PED-set, it follows that $\tilde{\mu}(G)\geq 1 + \mu(G)$, where $\mu(G)$ denotes the number of dominating induced matchings that $G$ admits (see \cite{Moyano-2017-thesis}).

\begin{rmk}~\label{rmk: connected components}
If $G_1, \ldots, G_k$ are the connected components of a graph $G$, then $\tilde{\mu}(G)=\displaystyle \prod_{i=1}^k \tilde{\mu}(G_i)$.
\end{rmk}

\begin{rmk}
If $G$ and $H$ are two isomorphic graphs, then $\tilde{\mu}(G)=\tilde{\mu}(H)$.
\end{rmk}

Let $G$ be a graph, and let $v$ be a vertex of $G$. Let $c$ be a color, which can be either black ($\mathsf{b}$), yellow ($\mathsf{y}$), or white ($\mathsf{w}$). We define the following three functions:
\begin{itemize}
     \item Let $\tilde{\mu}(G, v, c)$ denote the number of PED-sets of $G$ in which $v$ receives color $c$.
     \item Let $\varphi_v(G)$ denote the number of colorings of $G$ in which $v$ is colored yellow, all its neighbors in $V(G)$ are colored white, and the remaining vertices are colored so as to form a valid 3-coloring of $G-v$.
     \item $\varphi^*_v(G):=\tilde{\mu}(G)+\varphi_v(G)$.
\end{itemize}

\begin{rmk}~\label{rmk: vertex color}
Let $G$ be a graph and let $v$ be a vertex of $G$. Then it is straightforward to see that $\tilde{\mu}(G)=\tilde{\mu}(G, v, \mathsf{w})+\tilde{\mu}(G, v, \mathsf{y})+\tilde{\mu}(G, v, \mathsf{b})$.
\end{rmk}

As a consequence of the bijection between PED-sets and valid 3-colorings, the following lemma is immediate.

\begin{lem}
Given a graph $G$, counting the number of perfect edge dominating sets of $G$ is equivalent to counting all valid 3-colorings of $G$.
\end{lem}

Let $\mathcal{C}$ be a graph class. An \emph{extremal graph on $n$ vertices in $\mathcal{C}$} is a graph on $n$ vertices belonging to $\mathcal{C}$ that maximizes the number of perfect edge dominating sets among all graphs on $n$ vertices in $\mathcal{C}$.

\subsection{Perfect edge dominating sets in paths}~\label{subsec: paths}

In this subsection, we present a recursive formula for the number of PED-sets of the path on $n$ vertices, $P_n$.

\begin{prop}~\label{prop: ped of a graph plus P3}
Let $H$ be a graph and let $x$ be a leaf of $H$. Consider a graph $R$ isomorphic to $P_3$ with vertex set $V(R)=\{v_1,v_2,v_3\}$, where $v_1$ and $v_3$ are its leaves. Define $G'$ as the graph obtained from $H + R$ by adding the edge $v_3 x$, as shown in Figure \ref{Fig.: graph of prop 2.1}, and let $G:=G'-v_1$. Then $\tilde{\mu}(G')=\tilde{\mu}(G)+\tilde{\mu}(H)$.
\end{prop}

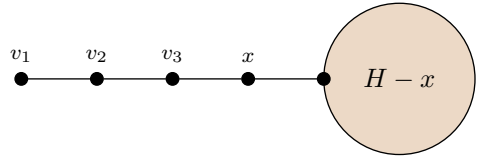
\begin{figure}
     \centering
     \begin{tikzpicture}[b/.style={circle, scale=0.5, draw=black, fill}, ye/.style={circle, scale=0.5, draw=black, fill=yellow}, w/.style={circle, scale=0.5, draw=black, fill=white}, line width=0.5pt]
          \draw[fill=brown!30] (0,0) circle (1);
          \node at (0,0) {$H-x$};
          \node (0) [b] at (-1,0){};
          \node (1) [b] at (-2,0){};
          \node[above] at (-2,0.1) {\footnotesize $x$};
          \node (2) [b] at (-3,0){};
          \node[above] at (-3,0.1) {\footnotesize $v_3$};
          \node (3) [b] at (-4,0){};
          \node[above] at (-4,0.1) {\footnotesize $v_2$};
          \node (4) [b] at (-5,0){};
          \node[above] at (-5,0.1) {\footnotesize $v_1$};
          \draw (4) -- (3) -- (2) -- (1) -- (0);
     \end{tikzpicture}
     \caption{Graph $G'$ described in Proposition \ref{prop: ped of a graph plus P3}}
     \label{Fig.: graph of prop 2.1}
\end{figure}

\begin{proof}
We compute $\tilde{\mu}(G')$ according to the color assigned to $v_1$. Since $v_1$ is a leaf of $G'$, it can be colored either white or yellow. Hence, $\tilde{\mu}(G')=\tilde{\mu}(G', v_1, \mathsf{w}) + \tilde{\mu}(G', v_1, \mathsf{y})$.

If $v_1$ is colored white, then $v_2$ must be colored yellow. Therefore, $\tilde{\mu}(G', v_1, \mathsf{w})=\tilde{\mu}(G, v_2, \mathsf{y})$.

Assume now that $v_1$ is colored yellow. Then $v_2$ must be assigned a non-white color. If $v_2$ is colored yellow, then $v_3$ must be colored white, and $x$ must be colored yellow. Moreover, the unique non-white neighbor of $x$ lies in $V(H)$. Therefore, the number of PED-sets of $G'$ in this case is $\tilde{\mu}(H, x, \mathsf{y})$.

On the other hand, if $v_2$ is colored black, then $v_3$ must be assigned a non-white color. If $v_3$ is colored yellow, then $x$ must be colored white, and the number of PED-sets of $G'$ in this case is $\tilde{\mu}(H, x, \mathsf{w})$. Finally, if $v_3$ is colored black, then $x$ is assigned a non-white color. Removing $v_1$, recoloring $v_2$ from black to white, and recoloring $v_3$ from black to yellow yield a valid 3-coloring for $G$. Therefore, the number of PED-sets of $G'$ in this case is $\tilde{\mu}(G, v_2, \mathsf{w})$. This change is reversible: starting from such a coloring of $G$ with $v_2$ white, one recovers the original coloring of $G'$ by adding $v_1$ as a vertex colored yellow and recoloring $v_2$ and $v_3$ black.

Therefore, $\tilde{\mu}(G', v_1, \mathsf{y})=\tilde{\mu}(H, x, \mathsf{y}) + \tilde{\mu}(H, x, \mathsf{w}) + \tilde{\mu}(G, v_2, \mathsf{w})$. Since $x$ is a leaf of $H$, $\tilde{\mu}(H, x, \mathsf{y}) + \tilde{\mu}(H, x, \mathsf{w})=\tilde{\mu}(H)$. Likewise, since $v_2$ is a leaf of $G$, $\tilde{\mu}(G, v_2, \mathsf{y}) + \tilde{\mu}(G, v_2, \mathsf{w})=\tilde{\mu}(G)$. Consequently, $\tilde{\mu}(G')=\tilde{\mu}(G) + \tilde{\mu}(H)$.
\end{proof}

The paths on a single vertex and on two vertices have only one PED-set. The path on three vertices has three PED-sets, whose associated 3-colorings are shown in the Figure \ref{Fig.: pedsets of P_3}. Therefore, $\tilde{\mu}(P_1)=1$, $\tilde{\mu}(P_2)=1$ and $\tilde{\mu}(P_3)=3$.

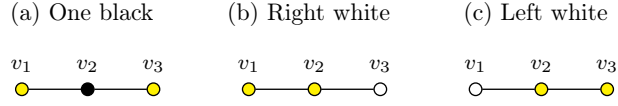
\begin{figure}
     \centering
     \begin{tikzpicture}[b/.style={circle, scale=0.5, draw=black, fill}, ye/.style={circle, scale=0.5, draw=black, fill=yellow}, w/.style={circle, scale=0.5, draw=black, fill=white}, line width=0.5pt]
          \node (0) [ye] at (-0.87,0){};
          \node[above] at (-0.87,0.1) {\footnotesize $v_1$};
	  \node (1) [b] at (0,0){};
          \node[above] at (0,0.1) {\footnotesize $v_2$};
	  \node (2) [ye] at (0.87,0){};
          \node[above] at (0.87,0.1) {\footnotesize $v_3$};
          \node (a)[scale=0.9] at (0-0.075,1){(a) One black};
	  \draw (0) -- (1) -- (2);
	  
	  \node (a0) [ye] at (-0.87+3,0){};
          \node[above] at (-0.87+3,0.1) {\footnotesize $v_1$};
	  \node (a1) [ye] at (0+3,0){};
          \node[above] at (0+3,0.1) {\footnotesize $v_2$};
	  \node (a2) [w] at (0.87+3,0){};
          \node[above] at (0.87+3,0.1) {\footnotesize $v_3$};
          \node (b)[scale=0.9] at (0-0.075+3,1){(b) Right white};
	  \draw (a0) -- (a1) -- (a2);
	  
	  \node (b0) [w] at (-0.87+6,0){};
          \node[above] at (-0.87+6,0.1) {\footnotesize $v_1$};
	  \node (b1) [ye] at (0+6,0){};
          \node[above] at (0+6,0.1) {\footnotesize $v_2$};
	  \node (b2) [ye] at (0.87+6,0){};
          \node[above] at (0.87+6,0.1) {\footnotesize $v_3$};
          \node (c)[scale=0.9] at (0-0.075+6,1){(c) Left white};
	  \draw (b0) -- (b1) -- (b2);
     \end{tikzpicture}
     \caption{All valid 3-colorings of the path on three vertices.}
     \label{Fig.: pedsets of P_3}
\end{figure}

\begin{theorem}~\label{thm: recursive formula}
If $n\geq 4$, then $\tilde{\mu}(P_n)=\tilde{\mu}(P_{n-1})+\tilde{\mu}(P_{n-3})$.
\end{theorem}

\begin{proof}
For $n=4$, a direct check gives $\tilde{\mu}(P_4)=4$, while $\tilde{\mu}(P_3)+\tilde{\mu}(P_1)=3+1=4$. For $n\geq 5$, it suffices to take $H=P_{n-3}$ in Proposition \ref{prop: ped of a graph plus P3} and to choose $x$ as one of its leaves.
\end{proof}

Consider the graphs $H$, $G$ and $G'$ of Proposition \ref{prop: ped of a graph plus P3}. If $v_1$ is colored white, then $v_2$ must be colored yellow, and $v_3$ is colored non-white. On the one hand, if $v_3$ is colored yellow, then $x$ must be colored white. On the other hand, if $v_3$ is colored black, then $x$ is assigned a non-white color. In the latter case, recoloring $v_2$ white and $v_3$ yellow yields the following remark.

\begin{rmk}~\label{rmk: first white vertex}
$\tilde{\mu}(G', v_1, \mathsf{w})=\tilde{\mu}(G, v_2, \mathsf{w})+\tilde{\mu}(H, x, \mathsf{w})$.
\end{rmk}

\section{Extremal problems}~\label{sec: extpro}

In this section, we identify the extremal graphs on $n$ vertices among trees, forests, and chordal graphs. We first treat trees and then consider the corresponding extremal problem for forests and chordal graphs.

\subsection{Perfect edge dominating sets in trees}~\label{subsec: pedtrees}

In this subsection, we prove that any tree $T$ on $n$ vertices has at most as many PED-sets as the path $P_n$. Moreover, we show that $P_n$ is the unique extremal tree when $n\not\in\{4, 5\}$; for $n=4$ and $n=5$, the additional extremal trees are $K_{1, 3}$, and $K_{1, 4}$, respectively.

\begin{lem}~\label{lem: pedstar}
Let $n\in\mathbb{N}$. Then $\tilde{\mu}(K_{1,1})=1$, and $\tilde{\mu}(K_{1,n-1})=n$ for all $n\ge 3$. Furthermore, $\tilde{\mu}(K_{1,n-1})=\tilde{\mu}(P_n)$ for $n\in\{2,3,4,5\}$, and $\tilde{\mu}(K_{1,n-1})<\tilde{\mu}(P_n)$ for all $n\ge 6$.
\end{lem}

\begin{proof}
Consider a graph $K$ isomorphic to the star on $n$ vertices, with $V(K)=\{u, v_1, \ldots, v_{n-1}\}$, where $u$ is the center of $K$ and the remaining $n-1$ vertices are leaves. Note that $u$ is assigned a non-white color. If $n=2$, then it is clear that $K$ has only one PED-set. Suppose $n\ge 3$. If $u$ is colored black, then all leaves $v_1, \ldots, v_{n-1}$ are colored yellow, resulting in the trivial PED-set. If $u$ is colored yellow, then exactly one leaf is colored yellow, while the remaining leaves are colored white, yielding $n-1$ PED-sets (in particular, these PED-sets are DIMs). Therefore, $K$ has a total of $n$ PED-sets.

For $n\in\{2,3,4,5\}$, it is easy to verify that $\tilde{\mu}(K_{1,n-1})=\tilde{\mu}(P_n)$. Finally, to prove that $\tilde{\mu}(K_{1,n-1})<\tilde{\mu}(P_n)$ for all $n\ge 6$, we proceed by induction on $n$. If $n=6$, $\tilde{\mu}(P_6)=8>6=\tilde{\mu}(K_{1,5})$, so the inequality holds. Suppose that $n>6$, and that $\tilde{\mu}(K_{1,n-2})<\tilde{\mu}(P_{n-1})$. By Theorem \ref{thm: recursive formula} and the inductive hypothesis, $\tilde{\mu}(P_n)>\tilde{\mu}(P_{n-1})>\tilde{\mu}(K_{1,n-2})=n-1$, which implies $\tilde{\mu}(P_n)>\tilde{\mu}(P_{n-1}) \ge n=\tilde{\mu}(K_{1,n-1})$.
\end{proof}

Let $Q$ be a graph isomorphic to $P_n$, with $V(Q)=\{v_1, \ldots, v_n\}$, where $v_1$ and $v_n$ are leaves and the remaining $n-2$ vertices have degree two. Define $q_n:=\varphi^*_{v_1}(Q)$, and set $q_0:=1$.

\begin{lem}~\label{lem: coloreos raros}
Let $\{q_n\}_{n\geq 0}$. Then the following conditions hold:
\begin{enumerate}
     \item $q_1=2$.
     \item $q_n=\tilde{\mu}(Q)+\tilde{\mu}(Q-v_1, v_2, \mathsf{w})$ for all $n\geq 2$.
     \item $q_n=q_{n-1}+q_{n-3}$ for all $n\geq 3$.
\end{enumerate}
\end{lem}

\begin{proof}
If $n=1$, then $V(Q)=\{v_1\}$ and $E(Q)=\emptyset$. Thus, $v_1$ can be colored white or yellow. Hence, $q_1=2$, and Statement 1 follows.

Statement 2 follows from the definition of the parameter $\varphi^*_v(G)$ when $G=Q$.

Finally, we prove Statement 3. We know that $q_0=1$ and $q_1=2$. Using Statement 2, $q_2=2$, $q_3=3$, and $q_4=5$, and hence $q_3=q_2+q_0$ and $q_4=q_3+q_1$. Assume $n>4$, and consider $G'=Q$, $G=G[\{v_2, v_3, \ldots, v_n\}]$, $H=G[\{v_4, \ldots, v_n\}]$, and $x=v_4$. By Statement 2, we have
$$q_{n-1}=\tilde{\mu}(G)+\tilde{\mu}(G-v_2, v_3, \mathsf{w}),\,\,\text{and}$$
$$q_{n-3}=\tilde{\mu}(H)+\tilde{\mu}(H-x, v_5, \mathsf{w}).$$
By applying Theorem \ref{thm: recursive formula}, we obtain $\tilde{\mu}(G)+\tilde{\mu}(H)=\tilde{\mu}(G')$, and by Remark \ref{rmk: first white vertex}, we have $\tilde{\mu}(G-v_2, v_3, \mathsf{w})+\tilde{\mu}(H-x, v_5, \mathsf{w})=\tilde{\mu}(G'-v_1, v_2, \mathsf{w})$. Consequently, $q_{n-1}+q_{n-3}=q_n$.
\end{proof}

Now, to achieve the goal of this subsection, we present two transformations that increase the number of PED-sets. The proofs of the following two propositions appear in Subsection \ref{subsec: proofsprop}.

\begin{prop}~\label{prop: pairs k and l}
Let $H$ be a connected graph, and let $v_0$ be a vertex of $H$ with $d_H(v_0)\geq 1$. Consider two graphs, $R$ and $S$, isomorphic to the paths $P_k$ and $P_l$, respectively, whose vertex sets are $V(R)=\{v_1, \ldots, v_k\}$ and $V(S)=\{u_1, \ldots, u_l\}$, where $k\leq l$. Define $G$ as the graph obtained from $H + R + S$ by adding the edges $v_0 v_1$ and $v_0 u_1$, and define $G'$ as the graph obtained from $H + R + S$ by adding the edges $v_0 v_1$ and $v_k u_1$, as shown in Figure \ref{Fig.: transformation one}. Then $\tilde{\mu}(G)<\tilde{\mu}(G')$, except when \textcolor{black}{$(k,l)\in\{(1,1), (1,2), (2,2), (2, 5)\}$. Moreover, when $(k,l)\in\{(1,2), (2,5)\}$, we have $\tilde{\mu}(G)\leq \tilde{\mu}(G')$.}
\end{prop}

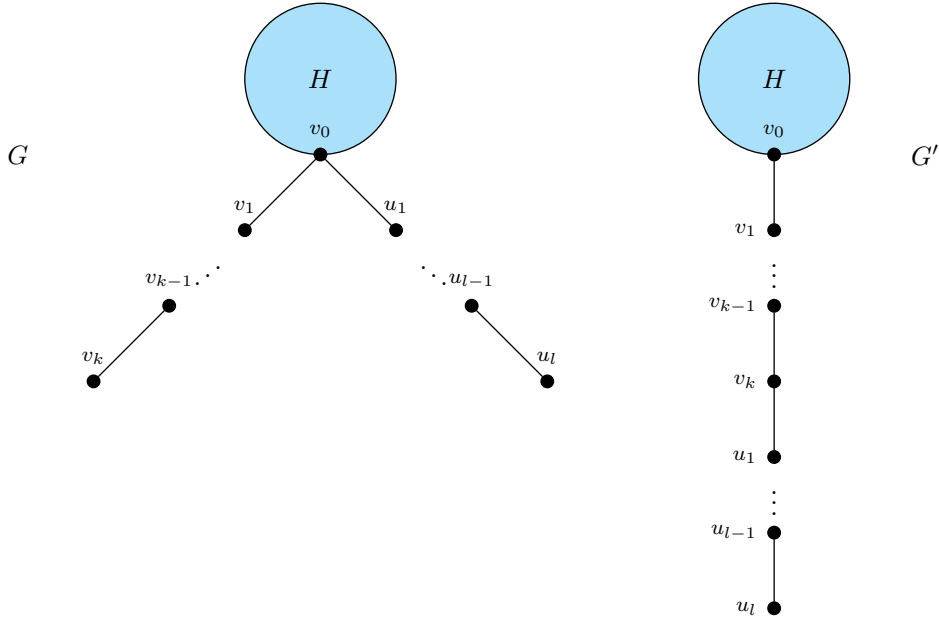
\begin{figure}
     \centering
     \begin{tikzpicture}[b/.style={circle, scale=0.5, draw=black, fill}, line width=0.5pt]
          \draw[fill=cyan!30] (0,0) circle (1);
          \node at (0,0) {$H$};
          \node (0) [b] at (0,-1){};
          \node[above] at (0,-0.9) {\footnotesize $v_0$};
          \node (1) [b] at (-1,-2){};
          \node[above] at (-1,-2+0.1) {\footnotesize $v_1$};
          \node (2) [b] at (-1-1,-2-1){};
          \node[above] at (-1-1,-2-1+0.1) {\footnotesize $v_{k-1}$};
          \node (3) [b] at (-1-2,-2-2){};
          \node[above] at (-1-2,-2-2+0.1) {\footnotesize $v_k$};
          \node (4) [b] at (1,-2){};
          \node[above] at (1,-2+0.1) {\footnotesize $u_1$};
          \node (5) [b] at (1+1,-2-1){};
          \node[above] at (1+1,-2-1+0.1) {\footnotesize $u_{l-1}$};
          \node (6) [b] at (1+2,-2-2){};
          \node[above] at (1+2,-2-2+0.1) {\footnotesize $u_l$};
          \draw (4) -- (0) -- (1) (2) -- (3) (5) -- (6);
          \node at (-1.5,-2.5) {$\iddots$};
          \node at (1.5,-2.5) {$\ddots$};
          \node at (-4,-1) {$G$};
          
          \draw[fill=cyan!30] (0+6,0) circle (1);
          \node at (6,0) {$H$};
          \node (0) [b] at (0+6,-1){};
          \node[above] at (0+6,-0.9) {\footnotesize $v_0$};
          \node (1) [b] at (6,-2){};
          \node[left] at (6-0.1,-2) {\footnotesize $v_1$};
          \node (2) [b] at (6,-3){};
          \node[left] at (6-0.1,-3) {\footnotesize $v_{k-1}$};
          \node (3) [b] at (6,-4){};
          \node[left] at (6-0.1,-4) {\footnotesize $v_k$};
          \node (4) [b] at (6,-5){};
          \node[left] at (6-0.1,-5) {\footnotesize $u_1$};
          \node (5) [b] at (6,-6){};
          \node[left] at (6-0.1,-6) {\footnotesize $u_{l-1}$};
          \node (6) [b] at (6,-7){};
          \node[left] at (6-0.1,-7) {\footnotesize $u_l$};
          \draw (0) -- (1) (2) -- (3) -- (4) (5) -- (6);
          \node at (6,-2.5) {$\vdots$};
          \node at (6,-5.5) {$\vdots$};
          \node at (8,-1) {$G'$};
     \end{tikzpicture}
     \caption{Transformation proposed in Proposition \ref{prop: pairs k and l}.}
     \label{Fig.: transformation one}
\end{figure}

Note that if $d_H(v_0)=0$ in Proposition \ref{prop: pairs k and l}, then the graphs $G$ and $G'$ are isomorphic, and the transformation is useless.

\begin{prop}~\label{prop: s-k}
Let $H$ be a connected graph and let $v_0$ be a leaf of $H$. Let $L$ be a graph isomorphic to $s P_1$ with vertex set $V(L)=\{l_1,\ldots,l_s\}$, and let $M$ be a graph isomorphic to $k P_2$ with vertex set $V(M)=\{u_1,v_1,\ldots,u_k,v_k\}$ and edge set $E(M)=\{u_1 v_1,\ldots,u_k v_k\}$, where $2k+s\geq 3$. Define $G$ as the graph obtained from $H + L + M$ by adding the edges $v_0 l_i$ for all $1\leq i\leq s$, and $v_0 u_j$ for all $1\leq j\leq k$, and define $G'$ as the graph obtained from $H + L + M$ by adding the edges $v_0 l_1, l_1 l_2, \ldots, l_{s-1} l_s, l_s u_1, v_1 u_2, v_2 u_3, \ldots, v_{k-1} u_k$ if $s\geq 1$, and by adding the edges $v_0 u_1, v_1 u_2, v_2 u_3, \ldots, v_{k-1} u_k$ if $s=0$, as shown in Figure \ref{Fig.: transformation two}. Then $\tilde{\mu}(G)<\tilde{\mu}(G')$, except when $(k, s)\in\{(0,3),(0,4)\}$.
\end{prop}

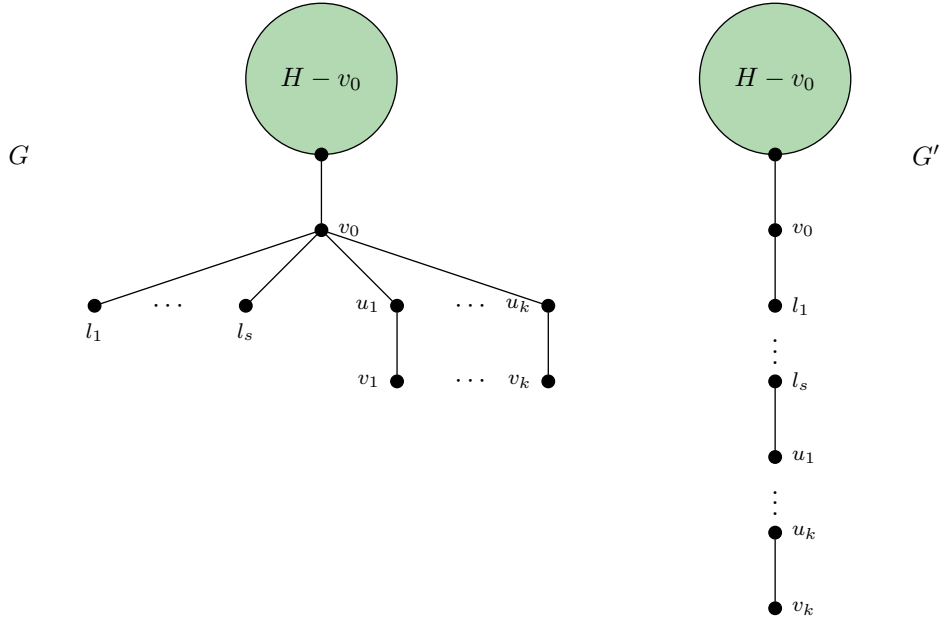
\begin{figure}
     \centering
     \begin{tikzpicture}[b/.style={circle, scale=0.5, draw=black, fill}, line width=0.5pt]
          \draw[fill=darkgreen!30] (0,0) circle (1);
          \node at (0,0) {$H-v_0$};
          \node (7) [b] at (0,-1){};
          \node (0) [b] at (0,-2){};
          \node[right] at (0.1,-2) {\footnotesize $v_0$};
          \node (1) [b] at (-3,-3){};
          \node[below] at (-3,-3-0.1) {\footnotesize $l_1$};
          \node (2) [b] at (-1,-3){};
          \node[below] at (-1,-3-0.1) {\footnotesize $l_s$};
          \node (3) [b] at (1,-3){};
          \node[left] at (1-0.1,-3) {\footnotesize $u_1$};
          \node (4) [b] at (3,-3){};
          \node[left] at (3-0.1,-3) {\footnotesize $u_k$};
          \node (5) [b] at (1,-4){};
          \node[left] at (1-0.1,-4) {\footnotesize $v_1$};
          \node (6) [b] at (3,-4){};
          \node[left] at (3-0.1,-4) {\footnotesize $v_k$};
          \draw (7) -- (0) -- (1) (0) -- (2) (5) -- (3) -- (0) -- (4) -- (6);
          \node at (-2,-3) {$\cdots$};
          \node at (2,-3) {$\cdots$};
          \node at (2,-4) {$\cdots$};
          \node at (-4,-1) {$G$};
          
          \draw[fill=darkgreen!30] (6,0) circle (1);
          \node at (6,0) {$H-v_0$};
          \node (7) [b] at (6,-1){};
          \node (0) [b] at (6,-2){};
          \node[right] at (6.1,-2) {\footnotesize $v_0$};
          \node (1) [b] at (6,-3){};
          \node[right] at (6.1,-3) {\footnotesize $l_1$};
          \node (2) [b] at (6,-4){};
          \node[right] at (6.1,-4) {\footnotesize $l_s$};
          \node (3) [b] at (6,-5){};
          \node[right] at (6.1,-5) {\footnotesize $u_1$};
          \node (4) [b] at (6,-6){};
          \node[right] at (6.1,-6) {\footnotesize $u_k$};
          \node (5) [b] at (6,-7){};
          \node[right] at (6.1,-7) {\footnotesize $v_k$};
          \draw (7) -- (0) -- (1) (2) -- (3) (4) -- (5);
          \node at (6,-3.5) {$\vdots$};
          \node at (6,-5.5) {$\vdots$};
          \node at (8,-1) {$G'$};
     \end{tikzpicture}
     \caption{Transformation proposed in Proposition \ref{prop: s-k}.}
     \label{Fig.: transformation two}
\end{figure}

Proposition \ref{prop: s-k} is stated in its full generality, although the case $k=0$ will not be used, since it is already covered in the proof of Theorem \ref{thm: pedtree}.

\begin{rmk}~\label{rmk: vertex util}
Given a tree $T$ on $n\geq 3$ vertices, there exists a vertex $x$ of it that is not a leaf such that the connected components of $T-x$ are paths except possibly one of them.
\end{rmk}

\begin{theorem}~\label{thm: pedtree}
If $T$ is a tree on $n$ vertices, then $\tilde{\mu}(T) \leq \tilde{\mu}(P_n)$. Furthermore, $\tilde{\mu}(T)=\tilde{\mu}(P_n)$ if and only if
\begin{itemize}
     \item $T\in\{K_{1,3}, P_4\}$ if $n=4$,
     \item $T\in\{K_{1,4}, P_5\}$ if $n=5$, or
     \item $T=P_n$ otherwise.
\end{itemize}
\end{theorem}

\begin{proof}
Let $T_0$ be an extremal tree on $n$ vertices. It suffices to prove that $T_0$ is isomorphic to the path on $n$ vertices, or to the stars $K_{1,3}$ and $K_{1,4}$ when $n=4$ and $n=5$, respectively.

If $T_0$ is isomorphic to $P_n$, then the theorem holds.

Suppose that $T_0$ is not a path. Then $T_0$ contains a vertex of degree at least three. If $T_0$ is a star, then, by Lemma \ref{lem: pedstar}, $T_0 \in \{K_{1, 3}, K_{1, 4}\}$ if $n\in\{4, 5\}$, and a star on at least six vertices is not an extremal tree. Therefore, we can suppose that $T_0$ is a tree on $n\geq 5$ vertices that is not a star. Among all extremal trees satisfying these assumptions, choose $T_0$ with the minimum possible number of leaves.

Let $v_0$ be a vertex of $T_0$ that satisfies Remark \ref{rmk: vertex util}. Since $T_0$ is not a path, we can assume that $d_{T_0}(v_0)\geq 3$. By the choice of $v_0$, all connected components of $T_0-v_0$ are paths except possibly one. We first rule out the case in which every connected component of $T_0-v_0$ is a path.

Suppose first that every connected component of $T_0-v_0$ is a path. Then there must exist at least one of them of length at least two, since $T_0$ is not a star. If one of these paths has length at least three, then we argue as follows. If there are two path-components whose lengths $k\leq l$ satisfy
\[
(k,l)\notin\{(1,1),(1,2),(2,2),(2,5)\},
\]
then Proposition \ref{prop: pairs k and l} implies the existence of another tree on $n$ vertices with a number of PED-sets greater than $\tilde{\mu}(T_0)$, contradicting the definition of $T_0$.

It remains to consider the only case not covered by this argument. In that case, one connected component of $T_0-v_0$ is a path of length five, at least one connected component of $T_0-v_0$ is a path of length two, and all the remaining connected components of $T_0-v_0$ are paths of length at most two. Let $H$ be the subgraph induced by $v_0$ together with the vertices of the path of length five. Then $v_0$ is a leaf of $H$. The remaining components form $sP_1+rP_2$, with $r\geq 1$ and $2r+s\geq 3$. Hence, by Proposition \ref{prop: s-k}, there would exist another tree $T'$ on $n$ vertices with $\tilde{\mu}(T')>\tilde{\mu}(T_0)$, again a contradiction.

Thus, all path-components of $T_0-v_0$ have length at most two. If at least one of them has length two, choose one path-component to keep together with $v_0$ so that the remaining path-components contain at least three vertices. Let $H$ be the subgraph induced by $v_0$ together with the vertices of the chosen path-component. Then $v_0$ is a leaf of $H$, and the remaining components form $sP_1+rP_2$ with $2r+s\geq 3$. Hence, by Proposition \ref{prop: s-k}, there would exist another tree $T'$ on $n$ vertices with $\tilde{\mu}(T')>\tilde{\mu}(T_0)$, again a contradiction. Therefore, all connected components of $T_0-v_0$ would have length one, implying that $T_0$ is a star, a contradiction. This proves that the case in which every connected component of $T_0-v_0$ is a path is impossible.

Since, by the choice of $v_0$, there is at most one connected component of $T_0-v_0$ that is not a path, there must exist exactly one such component. %We now prove that all the remaining components have length one.

Suppose, to the contrary, that some path-component of $T_0-v_0$ has length at least two. If there are two path-components whose lengths $k\leq l$ satisfy
\[
(k,l)\notin\{(1,1),(1,2),(2,2),(2,5)\},
\]
then Proposition \ref{prop: pairs k and l} implies the existence of another tree on $n$ vertices with a number of PED-sets greater than $\tilde{\mu}(T_0)$, contradicting the definition of $T_0$.

It remains to consider the case in which the only obstruction to applying the preceding argument is the pair $(2,5)$. In this case, the transformation of Proposition \ref{prop: pairs k and l} gives a tree $T'$ on $n$ vertices such that
\[
\tilde{\mu}(T')\geq \tilde{\mu}(T_0).
\]
If the inequality is strict, we again obtain a contradiction. If equality holds, then $T'$ is also extremal. Moreover, the unique non-path component of $T_0-v_0$ remains unchanged, and the two path-components of lengths two and five are replaced by a single path-component. Hence $T'$ is neither a path nor a star, and it has fewer leaves than $T_0$, contradicting the choice of $T_0$. Therefore, no path-component of $T_0-v_0$ has length at least three.

Thus, all path-components of $T_0-v_0$ have length at most two. If one of them has length two, let $H$ be the subgraph induced by $v_0$ together with the vertices of the unique connected component of $T_0-v_0$ that is not a path. Then $v_0$ is a leaf of $H$. The remaining components form $sP_1+rP_2$, with $r\geq 1$ and $2r+s\geq 3$. Hence, by Proposition \ref{prop: s-k}, there would exist another tree $T'$ on $n$ vertices with $\tilde{\mu}(T')>\tilde{\mu}(T_0)$, again a contradiction. Therefore, all the remaining components have length one, that is, they are leaves of $T_0$.

Let $L=\{l_1, l_2, \ldots, l_s\}\subseteq V(T_0)$ be the set of leaves adjacent to $v_0$, with $s\geq 2$, and let $y_0$ be the unique neighbor of $v_0$ that is not a leaf of $T_0$. We claim that $\tilde{\mu}(T_0)<\tilde{\mu}(P_n)$ for all $n\geq 5$. We proceed by induction on $n$.

If $n=5$, then $V(T_0)={l_1, l_2, v_0, y_0, z_0}$, where $z_0$ is a leaf adjacent to $y_0$. If we consider the subgraph $H=T_0[{v_0, l_1}]$, then $v_0$ is a leaf in $H$, and is adjacent to the paths induced by $l_2$ and ${y_0, z_0}$. By Proposition \ref{prop: s-k}, $\tilde{\mu}(T_0)<\tilde{\mu}(P_5)$.

Now, suppose that $n>5$, and that the theorem holds for every tree on $i<n$ vertices. Since $l_1$ is a leaf of $T_0$, it can be colored white or yellow. If $l_1$ is colored white, then $v_0$ must be colored yellow, and thus, $\tilde{\mu}(T_0, l_1, \mathsf{w}) = \tilde{\mu}(T_0-l_1, v_0, \mathsf{y})$. On the other hand, if $l_1$ is colored yellow, then $v_0$ must be colored yellow or black, and thus, $\tilde{\mu}(T_0, l_1, \mathsf{y}) = \tilde{\mu}(T_0-l_1, v_0, \mathsf{b}) + \tilde{\mu}(T_0-\{l_1, l_2, \ldots, l_s, v_0\}, y_0, \mathsf{w})$. Note that $v_0$ cannot be colored white since it is adjacent to at least one leaf. Furthermore, the graph $T_0-\{l_1, l_2, \ldots, l_s, v_0\}$ admits the trivial PED-set, and there exists a valid 3-coloring of it in which $y_0$ is assigned a non-white color. Thus, $\tilde{\mu}(T_0-\{l_1, l_2, \ldots, l_s, v_0\}, y_0, \mathsf{w}) < \tilde{\mu}(T_0-\{l_1, l_2, \ldots, l_s, v_0\})$. Finally, since $T_0-l_1$ is a tree on $n-1$ vertices, by the inductive hypothesis, $\tilde{\mu}(T_0-l_1) \leq \tilde{\mu}(P_{n-1})$. Likewise, since $T_0-\{l_1, l_2, \ldots, l_s, v_0\}$ is a tree on $n-s-1$ vertices, and $n-s-1 \leq n-3$, by the inductive hypothesis, $\tilde{\mu}(T_0-\{l_1, l_2, \ldots, l_s, v_0\}) \leq \tilde{\mu}(P_{n-3})$. To sum up,
\begin{align*}
	\tilde{\mu}(T_0)&=\tilde{\mu}(T_0, l_1, \mathsf{w}) + \tilde{\mu}(T_0, l_1, \mathsf{y}) \\
	&=\tilde{\mu}(T_0-l_1, v_0, \mathsf{y}) + (\tilde{\mu}(T_0-l_1, v_0, \mathsf{b}) + \tilde{\mu}(T_0-\{l_1, l_2, \ldots, l_s, v_0\}, y_0, \mathsf{w})) \\
	&=(\tilde{\mu}(T_0-l_1, v_0, \mathsf{y}) + \tilde{\mu}(T_0-l_1, v_0, \mathsf{b})) + \tilde{\mu}(T_0-\{l_1, l_2, \ldots, l_s, v_0\}, y_0, \mathsf{w}) \\
	&<\tilde{\mu}(T_0-l_1) + \tilde{\mu}(T_0-\{l_1, l_2, \ldots, l_s, v_0\}) \\
	&\leq \tilde{\mu}(P_{n-1}) + \tilde{\mu}(P_{n-3}) \\
	&=\tilde{\mu}(P_n).
\end{align*}

Consequently, $\tilde{\mu}(T_0)<\tilde{\mu}(P_n)$ for all $n\geq 5$. This contradicts the maximality of $\tilde{\mu}(T_0)$. In conclusion, $T_0$ must be isomorphic to the path $P_n$, or to the stars $K_{1,3}$ or $K_{1,4}$ when $n=4$ or $n=5$, respectively.

\end{proof}

\subsection{Proofs of Propositions \ref{prop: pairs k and l} and \ref{prop: s-k}.}~\label{subsec: proofsprop}

In this subsection, we prove the Propositions \ref{prop: pairs k and l} and \ref{prop: s-k}. We begin with some preliminary observations. Let $M$ be a graph, let $v_0\in V(M)$, and let $Q$ be a graph isomorphic to $P_k$, with vertex set $V(Q)=\{z_1, \ldots, z_k\}$, where $z_1$ and $z_k$ are the leaves of $Q$. Let $G$ be the graph obtained from $M+Q$ by adding the edge $v_0 z_1$. Assume first that $d_M(v_0)\geq 2$. Then
$$\tilde{\mu}(G)=q_{k-2} \tilde{\mu}(M, v_0, \mathsf{w}) + q_{k-1} \tilde{\mu}(M, v_0, \mathsf{b}) + q_{k-3} \tilde{\mu}(M, v_0, \mathsf{y}) + q_{k-1} \varphi_{v_0}(M).$$

Suppose that a vertex $z_j$ of $Q$ is assigned a non-white color, and consider the induced path $Q_j:=G[\{z_{j+1}, \ldots, z_k\}]$. Since $z_{j+1}$ is a leaf of $Q_j$, it can be either colored white or yellow. If $z_{j+1}$ is colored yellow, then its unique non-white neighbor may be either $z_j$ or $z_{j+2}$. In the former case, $z_{j+2}$ must be colored white, and hence $z_{j+1}$ is a vertex colored yellow whose neighbors in $Q_j$ are all colored white. Therefore, the number of colorings of $Q_j$ is $q_{k-j}$.

We compute $\tilde{\mu}(G)$. If $v_0$ is colored white, then all its neighbors in $M$ are colored yellow. Moreover, $z_1$ must also be colored yellow, and $z_2$ must be assigned a non-white color. Therefore, $\tilde{\mu}(G, v_0, \mathsf{w})=q_{k-2} \tilde{\mu}(M, v_0, \mathsf{w})$. If $v_0$ is colored black, then, since $d_M(v_0)\geq 2$, it has at least two neighbors in $M$, and all of which are colored non-white. Moreover, $z_1$ must be assigned a non-white color. Therefore, we have $\tilde{\mu}(G, v_0, \mathsf{b})=q_{k-1} \tilde{\mu}(M, v_0, \mathsf{b})$. Finally, if $v_0$ is colored yellow, then its unique neighbor colored non-white lies either in $M$ or outside $M$. In the former case, $z_1$ must be colored white, $z_2$ yellow, and $z_3$ assigned a non-white color. In the latter case, $z_1$ is assigned a non-white color. Therefore, $\tilde{\mu}(G, v_0, \mathsf{y})=q_{k-3} \tilde{\mu}(M, v_0, \mathsf{y}) + q_{k-1} \varphi_{v_0}(M)$. Combining the three cases above yields the desired expression for $\tilde{\mu}(G)$.

Now suppose that $d_M(v_0)=1$. The only difference from the previous case arises when $v_0$ is colored black. Indeed, the unique neighbor of $v_0$ in $M$ must be colored non-white, and therefore recoloring $v_0$ from black to yellow produces a valid 3-coloring of $M$. Consequently, $\tilde{\mu}(G, v_0, \mathsf{b})=q_{k-1} \tilde{\mu}(M, v_0, \mathsf{y})$.

\begin{proof}[Proof of Proposition \ref{prop: pairs k and l}]
Suppose first that $d_H(v_0)\geq 2$. By applying the above remark twice, the number of PED-sets of $G$ is
$$\tilde{\mu}(G)=q_{k-2} \, q_{l-2} \tilde{\mu}(H, v_0, \mathsf{w}) + q_{k-1} \, q_{l-1} \tilde{\mu}(H, v_0, \mathsf{b}) + q_{k-3} \, q_{l-3} \tilde{\mu}(H, v_0, \mathsf{y}) + (q_{k-1} \, q_{l-3} + q_{k-3} \, q_{l-1})\varphi_{v_0}(H),$$
whereas the number of PED-sets of $G'$ is:
$$\tilde{\mu}(G')=q_{k+l-2} \tilde{\mu}(H, v_0, \mathsf{w}) + q_{k+l-1} \tilde{\mu}(H, v_0, \mathsf{b}) + q_{k+l-3} \tilde{\mu}(H, v_0, \mathsf{y}) + q_{k+l-1} \varphi_{v_0}(H).$$

Recall that $k\leq l$. The following four claims hold for every $k\geq 3$ and $l\geq 3$, whose proofs proceed by induction on $l$.
\medskip

\noindent \textbf{Claim 1.} $q_{k+l-2}>q_{k-2} q_{l-2}$.
\medskip

\emph{Proof of the claim.} 
%If $l=4$, then, if $k=2$, $q_4 > q_0 q_2$; if $k=3$, $q_5 > q_1 q_2$; and if $k=4$, $q_6 > q_2 q_2$. Suppose $l>4$, and that Claim 1 holds for values fewer than $l$. Then
If $l=3$, then $k=3$, and $q_4>q_1 q_1$. If $l=4$, then, if $k=3$, $q_5 > q_1 q_2$; and if $k=4$, $q_6 > q_2 q_2$. If $l=5$, then, if $k=3$, $q_6 > q_1 q_3$; if $k=4$, $q_7 > q_2 q_3$; and if $k=5$, $q_8 > q_3 q_3$. Suppose $l>5$, and that Claim 1 holds for values fewer than $l$. Then
\begin{align*}
     q_{k+l-2} &= q_{(k+l-2)-1} + q_{(k+l-2)-3} \\
     &=q_{k+(l-1)-2} + q_{k+(l-3)-2} \\
     &>q_{k-2} q_{(l-1)-2} + q_{k-2} q_{(l-3)-2} \\
     &=q_{k-2}(q_{(l-2)-1} + q_{(l-2)-3}) \\
     &=q_{k-2} q_{l-2}.
\end{align*}

\noindent \textbf{Claim 2.} $q_{k+l-1}>q_{k-1} q_{l-1}$.
\medskip

\emph{Proof of the claim.} 
%If $l=3$, then, if $k=1$, $q_3 > q_0 q_2$; if $k=2$, $q_4 > q_1 q_2$; and if $k=3$, $q_5 > q_2 q_2$. Suppose $l>3$, and that Claim 2 holds for values fewer than $l$. Then
If $l=3$, then $k=3$, and $q_5 > q_2 q_2$. If $l=4$, then, if $k=3$, $q_6 > q_2 q_3$; and if $k=4$, $q_7 > q_3 q_3$. If $l=5$, then, if $k=3$, $q_7 > q_2 q_4$; if $k=4$, $q_8 > q_3 q_4$; and if $k=5$, $q_9 > q_4 q_4$. Suppose $l>5$, and that Claim 2 holds for values fewer than $l$. Then
\begin{align*}
     q_{k+l-1} &= q_{(k+l-1)-1} + q_{(k+l-1)-3} \\
     &=q_{k+(l-1)-1} + q_{k+(l-3)-1} \\
     &>q_{k-1} q_{(l-1)-1} + q_{k-1} q_{(l-3)-1} \\
     &=q_{k-1}(q_{(l-1)-1} + q_{(l-1)-3}) \\
     &=q_{k-1} q_{l-1}.
\end{align*}

%\textcolor{black}{Moreover, equality in the above inequality would force equality simultaneously for the pairs $(k,l-1)$ and $(k,l-3)$. By the induction hypothesis, this would imply $(k,l-1)=(2,5)$ and $(k,l-3)=(2,5)$, which is impossible. Hence, for $l>5$, the inequality is strict, and the only equality case is $(k,l)=(2,5)$.}

\noindent \textbf{Claim 3.}
 $q_{k+l-3}>q_{k-3} q_{l-3}$.
\medskip

\emph{Proof of the claim.} 
%If $l=5$, then, if $k=3$, $q_5 > q_0 q_2$; if $k=4$, $q_6 > q_1 q_2$; and if $k=5$, $q_7 > q_2 q_2$. Suppose $l>5$, and that Claim 3 holds for values fewer than $l$. Then
If $l=3$, then $k=3$, and $q_3>q_0 q_0$. If $l=4$, then, if $k=3$, $q_4>q_0 q_1$; and if $k=4$, $q_5>q_1 q_1$. If $l=5$, then, if $k=3$, $q_5 > q_0 q_2$; if $k=4$, $q_6 > q_1 q_2$; and if $k=5$, $q_7 > q_2 q_2$. Suppose $l>5$, and that Claim 3 holds for values fewer than $l$. Then
\begin{align*}
     q_{k+l-3} &= q_{(k+l-3)-1} + q_{(k+l-3)-3} \\
     &=q_{k+(l-1)-3} + q_{k+(l-3)-3} \\
     &>q_{k-3} q_{(l-1)-3} + q_{k-3} q_{(l-3)-3} \\
     &=q_{k-3}(q_{(l-3)-1} + q_{(l-3)-3}) \\
     &=q_{k-3} q_{l-3}
\end{align*}

\noindent \textbf{Claim 4.} $q_{k+l-1}>q_{k-1} q_{l-3} + q_{k-3} q_{l-1}$.
\medskip

\emph{Proof of the claim.} 
%If $l=5$, then, if $k=3$, $q_7 > q_2 q_2 + q_0 q_4$; if $k=4$, $q_8 > q_3 q_2 + q_1 q_4$; and if $k=5$, $q_9 > q_4 q_2 + q_2 q_4$. Suppose $l>5$, and that Claim 4 holds for values fewer than $l$. Then
If $l=3$, then $k=3$, and $q_5>q_2 q_0 + q_0 q_2$. If $l=4$, then, if $k=3$, $q_6>q_2 q_1 + q_0 q_3$; and if $k=4$, $q_7>q_3 q_1 + q_1 q_3$. If $l=5$, then, if $k=3$, $q_7 > q_2 q_2 + q_0 q_4$; if $k=4$, $q_8 > q_3 q_2 + q_1 q_4$; and if $k=5$, $q_9 > q_4 q_2 + q_2 q_4$. Suppose $l>5$, and that Claim 4 holds for values fewer than $l$. Then
\begin{align*}
     q_{k+l-1} &= q_{(k+l-1)-1} + q_{(k+l-1)-3} \\
     &=q_{k+(l-1)-1} + q_{k+(l-3)-1} \\
     &>(q_{k-1} q_{(l-1)-3} + q_{k-3} q_{(l-1)-1}) + (q_{k-1} q_{(l-3)-3} + q_{k-3}q_{(l-3)-1}) \\
     &=q_{k-1}(q_{(l-3)-1} + q_{(l-3)-3}) + q_{k-3}(q_{(l-1)-1} + q_{(l-1)-3}) \\
     &=q_{k-1} q_{l-3} + q_{k-3} q_{l-1}
\end{align*}

The following two claims analyze $\tilde{\mu}(G)$ and $\tilde{\mu}(G')$ when $k=1$ and $k=2$, respectively, that is, when $R$ is isomorphic to $P_1$ or $P_2$.
\medskip

\noindent \textbf{Claim 5.} If $k=1$ and $l\geq 3$, the number of PED-sets of $G$ is
$$\tilde{\mu}(G)=0\cdot\tilde{\mu}(H, v_0, \mathsf{w}) + q_{l-1} \tilde{\mu}(H, v_0, \mathsf{b}) + q_{l-3} \tilde{\mu}(H, v_0, \mathsf{y}) + (q_{l-3} + q_{l-1}) \varphi_{v_0}(H),$$
whereas the number of PED-sets of $G'$ is
$$\tilde{\mu}(G')=q_{l-1} \tilde{\mu}(H, v_0, \mathsf{w}) + q_l \tilde{\mu}(H, v_0, \mathsf{b}) + q_{l-2} \tilde{\mu}(H, v_0, \mathsf{y}) + q_l \varphi_{v_0}(H).$$

Then $\tilde{\mu}(G) < \tilde{\mu}(G')$.
\medskip

\emph{Proof of the claim.} According to the definition of the succession $(q_n)_{n\geq 0}$ and by Lemma \ref{lem: coloreos raros}, it follows that $q_{l-1} > 0$, $q_l > q_{l-1}$, $q_{l-2} > q_{l-3}$, and $q_l = q_{l-3} + q_{l-1}$. Due to the existence of the trivial PED-set of $H$, $\tilde{\mu}(H, v_0, \mathsf{b})$ or $\tilde{\mu}(H, v_0, \mathsf{y})$ is non-zero. Therefore, $\tilde{\mu}(G) < \tilde{\mu}(G')$.
\medskip

\noindent \textbf{Claim 6.} If $k=2$ and $l\geq 3$, the number of PED-sets of $G$ is
$$\tilde{\mu}(G)=q_{l-2} \tilde{\mu}(H, v_0, \mathsf{w}) + q_1 q_{l-1} \tilde{\mu}(H, v_0, \mathsf{b}) + 0\cdot\tilde{\mu}(H, v_0, \mathsf{y}) + q_1 q_{l-3} \varphi_{v_0}(H),$$
whereas the number of PED-sets of $G'$ is
$$\tilde{\mu}(G')=q_l \tilde{\mu}(H, v_0, \mathsf{w}) + q_{l+1} \tilde{\mu}(H, v_0, \mathsf{b}) + q_{l-1} \tilde{\mu}(H, v_0, \mathsf{y}) + q_{l+1} \varphi_{v_0}(H).$$

Then $\tilde{\mu}(G) < \tilde{\mu}(G')$, except when $l=5$, in whose case $\tilde{\mu}(G) \leq \tilde{\mu}(G')$.
\medskip

\emph{Proof of the claim.} The same argument of the above proof shows that $q_l > q_{l-2}$ and $q_{l-1} > 0$. What remains to be shown is that $q_{l+1} > q_1 q_{l-1}$ and $q_{l+1} > q_1 q_{l-3}$. Note that $q_{l+1} > q_1 q_{l-1}$ implies $q_{l+1} > q_1 q_{l-3}$ since $q_{l-1} > q_{l-3}$. Therefore, it suffices to prove the first inequality.  We proceed by induction on $l$. If $l=3$, $q_4>q_1 q_2$. If $l=4$, $q_5>q_1 q_3$. If $l=5$, $q_6=q_1 q_4$, and this shows that $\tilde{\mu}(G) \leq \tilde{\mu}(G')$. If $l=6$, $q_7>q_1 q_5$. If $l=7$, $q_8>q_1 q_6$. And if $l=8$, $q_9>q_1 q_7$. Suppose that $l>8$, and that Claim 6 holds for values fewer than $l$. Then
$$q_{l+1}=q_{(l+1)-1} + q_{(l+1)-3}=q_{(l-1)+1} + q_{(l-3)+1}>q_1 q_{(l-1)-1} + q_1 q_{(l-3)-1}=q_1 q_{l-1}.$$
%\textcolor{black}{By Claim 2, we have $q_{l+1}\geq q_1 q_{l-1}$, and equality holds only when $l=5$. Since $l=5$ is excluded in this claim, we have $q_{l+1}>q_1 q_{l-1}$.} 
Therefore, $\tilde{\mu}(G) < \tilde{\mu}(G')$.
\medskip

Now, suppose that $d_H(v_0)=1$. The number of PED-sets of $G$ is
$$\tilde{\mu}(G)=q_{k-2} q_{l-2} \tilde{\mu}(H, v_0, \mathsf{w}) + (q_{k-1} q_{l-1} + q_{k-3} q_{l-3}) \tilde{\mu}(H, v_0, \mathsf{y}) + (q_{k-1} q_{l-3} + q_{k-3} q_{l-1}) \varphi_{v_0}(H),$$
whereas the number of PED-sets of $G'$ is
$$\tilde{\mu}(G')=q_{k+l-2} \tilde{\mu}(H, v_0, \mathsf{w}) + q_{k+l} \tilde{\mu}(H, v_0, \mathsf{y}) + q_{k+l-1} \varphi_{v_0}(H).$$

The case $(k,l)=(2,5)$ is the only additional case in which the coefficient comparison above may fail to be strict, because $q_6=q_1q_4$. This case is excluded in the statement. To prove that $\tilde{\mu}(G)<\tilde{\mu}(G')$, we argue as in the previous case.
\end{proof}

\begin{proof}[Proof of Proposition \ref{prop: s-k}]
The number of PED-sets of $G$ is
$$\tilde{\mu}(G)=t_s \tilde{\mu}(H, v_0, \mathsf{w}) + 2^k \tilde{\mu}(H, v_0, \mathsf{y}) + x_k \tilde{\mu}(H, v_0, \mathsf{y}) + z_k \varphi_{v_0}(H),$$
whereas the number of PED-sets of $G'$ is
$$\tilde{\mu}(G')=q_{2k+s-2} \tilde{\mu}(H, v_0, \mathsf{w}) + q_{2k+s-1} \tilde{\mu}(H, v_0, \mathsf{y}) + q_{2k+s-3} \tilde{\mu}(H, v_0, \mathsf{y}) + q_{2k+s-1} \varphi_{v_0}(H),$$
where
\begin{multicols}{3}
\[
t_s = \begin{cases}
    0 & \text{if } s > 0 \\
    1 & \text{if } s = 0
\end{cases}
\]

\[
x_k = \begin{cases}
    0 & \text{if } k > 0 \\
    1 & \text{if } k = 0
\end{cases}
\]

\[
z_k = \begin{cases}
    0 & \text{if } k > 1 \\
    \textcolor{black}{2} & \text{if } k = 1 \\
    s & \text{if } k = 0
\end{cases}
\]
\end{multicols}

The value $z_1=2$ follows from the fact that, when $k=1$ and $v_0$ is colored yellow with its unique non-white neighbor outside $H$, there are two possible extensions: either $u_1$ is colored yellow and $v_1$ is colored white, or $u_1$ is colored black and $v_1$ is colored yellow.

Thus, it suffices to prove that $q_{2k+s-2} > t_s$, $q_{2k+s} > 2^k + x_k$, and $q_{2k+s-1} \geq z_k$. Recall that $2k+s\geq 3$.

Since $2k+s-2\geq 1$, then $q_{2k+s-2}\geq q_1=2 > t_s$. This proves the first inequality.

To prove the second inequality, $q_{2k+s} > 2^k + x_k$, we proceed by induction on $k$. If $k=0$, then $q_s > 2^0 + x_0$; if $k=1$, then $q_{s+2} > 2^1 + x_1$; and if $k=2$, then $q_{s+4} > 2^2 + x_2$. Suppose that $k>2$, and that the inequality holds for each value less than $k$. Then, by the inductive hypothesis,
\begin{align*}
     q_{2k+s} &= q_{2k+s-1} + q_{2k+s-3} \\
     &= (q_{2k+s-2} + q_{2k+s-4}) + (q_{2k+s-4} + q_{2k+s-6}) \\
     &= q_{2(k-1)+s} + 2q_{2(k-2)+s} + q_{2(k-3)+s} \\
     &> 2^{k-1} + x_{k-1} + 2(2^{k-2} + x_{k-2}) + 2^{k-3} + x_{k-3} \\
     &\geq 2^k + 2^{k-3} \\
     &> 2^k.
\end{align*}

\textcolor{black}{Finally, it is easy to verify that $q_{2k+s-1} \geq z_k$ for all $(k,s)\notin\{(0,3),(0,4)\}$. Indeed, if $k>1$, then $z_k=0$; if $k=1$, then $z_k=2$ and, since $2k+s\geq 3$, we have $s\geq 1$, and hence $q_{2k+s-1}=q_{s+1}\geq q_2=2=z_k$; and if $k=0$, then $q_{2k+s-1}=q_{s-1}\geq s=z_k$ for every $s\geq 5$. Therefore, the last inequality follows.}
\end{proof}

\subsection{Perfect edge dominating sets in forests}~\label{subsec: pedforests}

In this subsection, we determine the extremal forests on $n$ vertices. Denote by $\mathcal{F}_n$ the family of extremal forests on $n$ vertices. Thus $\mathcal{F}_1 = \{P_1\}$, and $\mathcal{F}_2 = \{P_2, 2P_1\}$.

Note that, for $n\geq 3$, every forest in $\mathcal{F}_n$ must contain a connected component on at least three vertices. Otherwise, if there existed a forest $F_0 \in \mathcal{F}_n$ whose connected components all have at most two vertices, then $\tilde{\mu}(F_0)=1$, contradicting the definition of $\mathcal{F}_n$. Consequently, it follows that $\mathcal{F}_3 = \{P_3\}$, $\mathcal{F}_4 = \{P_4, K_{1,3}\}$, and $\mathcal{F}_5 = \{P_5, K_{1,4}\}$.

\begin{lem}~\label{lem: maxforest}
Let $n\geq 3$. Then the following statements hold:
\begin{enumerate}
     \item For every $F_1 \in \mathcal{F}_{n-1}$ and $F_2 \in \mathcal{F}_n$, $\tilde{\mu}(F_1)<\tilde{\mu}(F_2)$.
     \item Every connected component of every forest $F \in \mathcal{F}_n$ has at least three vertices.
\end{enumerate}
\end{lem}

\begin{proof}
If $n=3$, then every forest in $\mathcal{F}_2$ has fewer PED-sets than every forest in $\mathcal{F}_3$. Assume now that $n\geq 4$, and let $F_1 \in \mathcal{F}_{n-1}$ and $F_2 \in \mathcal{F}_n$. Let $X$ be a connected component of $F_1$ with the largest number vertices, say $k$. Then $k\geq 3$. By Theorem \ref{thm: pedtree}, we have $\tilde{\mu}(X)\leq\tilde{\mu}(P_k)$, and by Theorem \ref{thm: recursive formula}, we have $\tilde{\mu}(P_k)<\tilde{\mu}(P_{k+1})$. Therefore,
$$\tilde{\mu}(F_1) = \tilde{\mu}(X) \tilde{\mu}(F_1-X) \leq \tilde{\mu}(P_k) \tilde{\mu}(F_1-X) < \tilde{\mu}(P_{k+1}) \tilde{\mu}(F_1-X) \leq \tilde{\mu}(F_2),$$
and Statement 1 follows.

Let $F \in \mathcal{F}_n$. Suppose, for contradiction, that $C$ is a connected component of $F$ on at most two vertices. Then $C$ is isomorphic to $P_1$ or $P_2$, and hence $\tilde{\mu}(C)=1$. Since $F-C$ has at most $n-1$ vertices, Statement 1 implies that $\tilde{\mu}(F-C)\leq \tilde{\mu}(F')<\tilde{\mu}(F)$ for every $F'\in \mathcal{F}_{n-1}$. On the other hand, $\tilde{\mu}(F) = \tilde{\mu}(C) \, \tilde{\mu}(F-C) = \tilde{\mu}(F-C)$, which is a contradiction.
\end{proof}

\begin{rmk}~\label{rmk: forests on six-twelve vertices}
Using Statement 2 of Lemma \ref{lem: maxforest}, we conclude that $\mathcal{F}_6 = \{2P_3\}$ and $\tilde{\mu}(2P_3)=9$, $\mathcal{F}_7 = \{P_7, P_3 + P_4, P_3 + K_{1,3}\}$ and $\tilde{\mu}(F)=12$ for all $F \in \mathcal{F}_7$, $\mathcal{F}_8 = \{P_8\}$ and $\tilde{\mu}(P_8)=17$, $\mathcal{F}_9 = \{3P_3\}$ and $\tilde{\mu}(3P_3)=27$, $\mathcal{F}_{10} = \{P_{10}\}$ and $\tilde{\mu}(P_{10})=37$, $\mathcal{F}_{11} = \{P_{11}\}$ and $\tilde{\mu}(P_{11})=54$, and $\mathcal{F}_{12} = \{4P_3\}$ and $\tilde{\mu}(4P_3)=81$.
\end{rmk}

\begin{theorem}~\label{thm: maxnumber pedforest}
If $n\geq 13$, then $\mathcal{F}_n = \{P_n\}$.
\end{theorem}

\begin{proof}
If $F$ is any forest on $n$ vertices with connected components $F_1,\ldots,F_k$, by Theorem \ref{thm: pedtree} we have $\tilde{\mu}(F)=\tilde{\mu}(F_1)\cdots\tilde{\mu}(F_k)\leq\tilde{\mu}(L_1)\cdots\tilde{\mu}(L_k)=\tilde{\mu}(L)$, where $L=\sum_{i=1}^k L_i$ is a linear forest. Suppose that $L$ has $n\geq 13$ vertices, and let $n_i=|V(L_i)|$ for all $1\leq i\leq k$. By Lemma \ref{lem: maxforest}, we may assume that $n_i\geq 3$ for all $1\leq i\leq k$.

We proceed by induction on $n$. Let $13\leq n\leq 25$. For any connected component $L'$ on $r$ vertices of $L$, we have $\tilde{\mu}(L) = \tilde{\mu}(L') \, \tilde{\mu}(L-L') \leq \tilde{\mu}(L') \, \tilde{\mu}(F') = \tilde{\mu}(L' + F')$, where $F' \in \mathcal{F}_{n-r}$. Table \ref{tab: tabla1} was constructed by the recurrence of Theorem \ref{thm: recursive formula}, and for each $3\leq r\leq 22$, the row labeled $r$ gives the largest value obtainable when one component has $r$ vertices and the remaining $n-r$ vertices form an extremal forest. It is clear that $\tilde{\mu}(L' + F')<\tilde{\mu}(P_n)$ for each $3\leq r\leq 22$.

\begin{table}[ht]
\centering
\begin{tabular}{|c|c|c|c|c|c|c|c|c|c|c|c|c|c|}
\hline
$n$ & $13$ & $14$ & $15$ & $16$ & $17$ & $18$ & $19$ & $20$ & $21$ & $22$ & $23$ & $24$ & $25$ \\
\hline
$\tilde{\mu}(P_n)$ & $116$ & $170$ & $249$ & $365$ & $535$ & $784$ & $1149$ & $1684$ & $2468$ & $3617$ & $5301$ & $7769$ & $11386$ \\
\hline
$3$ & $111$ & $162$ & $243$ & $348$ & $510$ & $747$ & $1095$ & $1605$ & $2352$ & $3447$ & $5052$ & $7404$ & $10851$ \\
\hline
$4$ & $108$ & $148$ & $216$ & $324$ & $464$ & $680$ & $996$ & $1460$ & $2140$ & $3136$ & $4596$ & $6736$ & $9872$ \\
\hline
$5$ & $85$ & $135$ & $185$ & $270$ & $405$ & $580$ & $850$ & $1245$ & $1825$ & $2675$ & $3920$ & $5745$ & $8420$ \\
\hline
$6$ & $96$ & $136$ & $216$ & $296$ & $432$ & $648$ & $928$ & $1360$ & $1992$ & $2920$ & $4280$ & $6272$ & $9192$ \\
\hline
$7$ & $108$ & $144$ & $204$ & $324$ & $444$ & $648$ & $972$ & $1392$ & $2040$ & $2988$ & $4380$ & $6420$ & $9408$ \\
\hline
$8$ & $85$ & $153$ & $204$ & $289$ & $459$ & $629$ & $918$ & $1377$ & $1972$ & $2890$ & $4233$ & $6205$ & $9095$ \\
\hline
$9$ & $100$ & $125$ & $225$ & $300$ & $425$ & $675$ & $925$ & $1350$ & $2025$ & $2900$ & $4250$ & $6225$ & $9125$ \\
\hline
$10$ & $111$ & $148$ & $185$ & $333$ & $444$ & $629$ & $999$ & $1369$ & $1998$ & $2997$ & $4292$ & $6290$ & $9213$ \\
\hline
$11$ & $54$ & $162$ & $216$ & $270$ & $486$ & $648$ & $918$ & $1458$ & $1998$ & $2916$ & $4374$ & $6264$ & $9180$ \\
\hline
$12$ & $79$ & $79$ & $237$ & $316$ & $395$ & $711$ & $948$ & $1343$ & $2133$ & $2923$ & $4266$ & $6399$ & $9164$ \\
\hline
$13$ & $-$ & $116$ & $116$ & $348$ & $464$ & $580$ & $1044$ & $1392$ & $1972$ & $3132$ & $4292$ & $6264$ & $9396$ \\
\hline
$14$ & $-$ & $-$ & $170$ & $170$ & $510$ & $680$ & $850$ & $1530$ & $2040$ & $2890$ & $4590$ & $6290$ & $9180$ \\
\hline
$15$ & $-$ & $-$ & $-$ & $249$ & $249$ & $747$ & $996$ & $1245$ & $2241$ & $2988$ & $4233$ & $6723$ & $9213$ \\
\hline
$16$ & $-$ & $-$ & $-$ & $-$ & $365$ & $365$ & $1095$ & $1460$ & $1825$ & $3285$ & $4380$ & $6205$ & $9855$ \\
\hline
$17$ & $-$ & $-$ & $-$ & $-$ & $-$ & $535$ & $535$ & $1605$ & $2140$ & $2675$ & $4815$ & $6420$ & $9095$ \\
\hline
$18$ & $-$ & $-$ & $-$ & $-$ & $-$ & $-$ & $784$ & $784$ & $2352$ & $3136$ & $3920$ & $7056$ & $9408$ \\
\hline
$19$ & $-$ & $-$ & $-$ & $-$ & $-$ & $-$ & $-$ & $1149$ & $1149$ & $3447$ & $4596$ & $5745$ & $10341$ \\
\hline
$20$ & $-$ & $-$ & $-$ & $-$ & $-$ & $-$ & $-$ & $-$ & $1684$ & $1684$ & $5052$ & $6736$ & $8420$ \\
\hline
$21$ & $-$ & $-$ & $-$ & $-$ & $-$ & $-$ & $-$ & $-$ & $-$ & $2468$ & $2468$ & $7404$ & $9872$ \\
\hline
$22$ & $-$ & $-$ & $-$ & $-$ & $-$ & $-$ & $-$ & $-$ & $-$ & $-$ & $3617$ & $3617$ & $10851$ \\
\hline
\end{tabular}
\caption{For $13\leq n\leq 25$, row 2 lists the values of $\tilde{\mu}(P_n)$, whereas row $r$ lists the values of $\tilde{\mu}(L' + F')$ for $3 \leq r \leq 22$ and $F' \in \mathcal{F}_{n-r}$.}
\label{tab: tabla1}
\end{table}

Now, suppose that $n>25$, and let $L_j$ be a connected component of $L$ with the fewest vertices. Note that $\tilde{\mu}(L_j)=\tilde{\mu}(P_{n_j})$. Then the linear forest $L-L_j=\sum_{i\neq j} L_i$ has $n-n_j\geq 13$ vertices; otherwise, if $n-n_j\leq 12$, then $n_j=n-(n-n_j)\geq 26-12=14$, contradicting the minimality of $n_j$. Thus, by the inductive hypothesis, $\tilde{\mu}(L-L_j)<\tilde{\mu}(P_{n-n_j})$, and then
\begin{align*}
     \tilde{\mu}(L)&=\tilde{\mu}(L_j)\cdot \tilde{\mu}(L-L_j) \\
     &< \tilde{\mu}(P_{n_j}) \tilde{\mu}(P_{n-n_j}) \\
     &=\tilde{\mu}(P_{n_j}) \, (\tilde{\mu}(P_{n-n_j-1})+\tilde{\mu}(P_{n-n_j-3})) \\
     &=\tilde{\mu}(P_{n_j}) \, \tilde{\mu}(P_{n-n_j-1}) + \tilde{\mu}(P_{n_j}) \, \tilde{\mu}(P_{n-n_j-3}) \\
     &=\tilde{\mu}(P_{n_j} + P_{n-n_j-1}) + \tilde{\mu}(P_{n_j} + P_{n-n_j-3}) \\
     &< \tilde{\mu}(P_{n-1}) + \tilde{\mu}(P_{n-3}) \\
     &=\tilde{\mu}(P_n)
\end{align*}
\end{proof}

\subsection{Perfect edge dominating sets in chordal graphs}~\label{subsec: pedchordal}

In this subsection, we determine the extremal chordal graphs on \(n\) vertices. The main structural ingredient is a theorem showing that, in an extremal graph containing a triangle, such a triangle must appear as an entire connected component. The proof works both for extremality among all graphs and for extremality restricted to chordal graphs; in the latter case, chordality is used only to ensure that the transformed graphs remain in the same class.

A triangle \(R\) admits exactly two types of valid 3-colorings: either all its vertices are colored black, or one vertex is colored white and the other two are colored yellow. These correspond, respectively, to the trivial PED-set and to a DIM. In the latter case, \(R\) admits exactly three DIMs, each consisting of a single edge. Hence \(\tilde{\mu}(R)=4\).

\begin{lem} \label{lem: chordal_operations}
Let \(G\) be a chordal graph. Then the following graphs are also chordal:
\begin{enumerate}
    \item every induced subgraph of \(G\);
    \item every graph obtained from \(G\) by contracting a connected induced
    subgraph into a single vertex;
    \item the disjoint union of \(G\) with any chordal graph.
\end{enumerate}
\end{lem}

\begin{proof}
The first and third assertions are immediate. For the second one, note that
contracting a connected induced subgraph can be obtained by successively
contracting the edges of a spanning tree of it. Since chordal graphs are
closed under edge contraction, the resulting graph is chordal.
\end{proof}

\begin{rmk} \label{rmk: contraction_recoloring}
We shall use the following recoloring convention after contractions. Suppose
\(X\subseteq V(G)\) induces a connected subgraph, and that, for the
valid 3-colorings under consideration, no vertex of \(X\) is colored white and
no white vertex outside \(X\) has a neighbor in \(X\). Let \(G_{X, x}\) be the
graph obtained from \(G\) by contracting \(X\) into a single vertex \(x\).

Given one such valid 3-coloring of \(G\), we define a coloring of \(G_{X, x}\) as
follows. The contracted vertex \(x\) is colored white, yellow, or black
according as \(d_{G_{X, x}}(x)\) is \(0\), \(1\), or at least \(2\). The colors of
the remaining vertices are kept unchanged, except for the following possible
adjustment: if a vertex \(u\in V(G)\setminus X\) was colored black and has
degree \(1\) in \(G_{X, x}\), then \(u\) is recolored yellow.

The resulting coloring is a valid 3-coloring of \(G_{X, x}\). Indeed, no white
vertex outside \(X\) becomes adjacent to \(x\), by assumption. A yellow vertex
outside \(X\) remains a leaf of \(G_{X, x}-W\), since its non-white neighbors in
\(X\), if any, are contracted into the single vertex \(x\). Finally, if a
black vertex \(u\) becomes of degree \(1\) in \(G_{X, x}\), then all but possibly
one of its neighbors in \(G\) belonged to \(X\), and recoloring \(u\) yellow
does not affect the validity of the coloring on the other vertices. This
adjustment is determined by \(G\) and by \(X\), and therefore it preserves the
injectivity of the maps used below.
\end{rmk}

\begin{theorem} \label{thm: triangle}
Let \(\mathcal C\) be either the class of all graphs or the class of chordal
graphs. Let \(G\) be an extremal graph on \(n\) vertices in \(\mathcal C\).
If \(H\) is a connected component of \(G\) containing a triangle \(R\), then
\(H=R\).
\end{theorem}

\begin{proof}
Let \(V(R)=\{v_1,v_2,v_3\}\). Since \(H\) is a connected component of \(G\),
Remark~\ref{rmk: connected components} gives
\[
    \tilde{\mu}(G)=\tilde{\mu}(H)\tilde{\mu}(G-V(H)).
\]
Suppose, for a contradiction, that \(H\neq R\).

We shall construct, in each case, a graph \(G'\) on \(n\) vertices with
\(\tilde{\mu}(G')>\tilde{\mu}(G)\). If \(\mathcal C\) is the class of all
graphs, then \(G'\in\mathcal C\) is automatic. If \(\mathcal C\) is the class
of chordal graphs, then \(G'\in\mathcal C\) follows from
Lemma~\ref{lem: chordal_operations}, since the constructions below only use
induced subgraphs, contractions of connected induced subgraphs, and disjoint
unions with chordal graphs. Thus, in both settings, such a graph \(G'\) is a
valid competitor, contradicting the extremality of \(G\).

If no vertex of \(V(H)\setminus V(R)\) is adjacent to a vertex of \(R\), then
\(R\) is a connected component of \(H\), and therefore of \(G\), contrary to
the connectedness of \(H\) and to \(H\neq R\). Hence there must exist a
vertex of \(V(H)\setminus V(R)\) with at least one neighbor in \(V(R)\).
Choose a vertex \(s\in V(H)\setminus V(R)\) maximizing \(d_R(s)\).

We distinguish three cases according to the value of \(d_R(s)\).

\medskip

\noindent\textbf{Case 1: \(d_R(s)=3\).}
Then \(V(R)\cup\{s\}\) induces a clique on four vertices in \(H\). In every
valid 3-coloring of \(H\), all vertices of \(V(R)\cup\{s\}\) are colored
black. Let \(H^*\) be the graph obtained from \(H\) by contracting
\(V(R)\cup\{s\}\) into a single vertex. By
Remark~\ref{rmk: contraction_recoloring}, every valid 3-coloring of \(H\)
induces a valid 3-coloring of \(H^*\). Therefore,
\[
    \tilde{\mu}(H)\leq \tilde{\mu}(H^*).
\]
Let \(R'\) be a new triangle, disjoint from \(H^*\) and from \(G-V(H)\), and
define
\[
    G'=R' + H^* + (G-V(H)).
\]
Then \(G'\) is a graph on \(n\) vertices in \(\mathcal C\), and
\[
    \tilde{\mu}(G')
    =4\tilde{\mu}(H^*)\tilde{\mu}(G-V(H))
    \geq 4\tilde{\mu}(H)\tilde{\mu}(G-V(H))
    >\tilde{\mu}(G),
\]
a contradiction.

\medskip

\noindent\textbf{Case 2: \(d_R(s)=2\).}
Without loss of generality, assume that
\[
    N_H(s)\cap V(R)=\{v_1,v_2\}.
\]
Then \(V(R)\cup\{s\}\) induces a diamond in \(H\). In a valid 3-coloring of
this diamond, there are only two possible patterns:
\begin{enumerate}
    \item all vertices of \(V(R)\cup\{s\}\) are colored black;
    \item \(s\) and \(v_3\) are colored white, while \(v_1\) and \(v_2\)
    are colored yellow.
\end{enumerate}
Thus
\[
    \tilde{\mu}(H)=\tilde{\mu}(H, s, \mathsf{b})
    +\tilde{\mu}(H, s, \mathsf{w}).
\]

If \(\tilde{\mu}(H, s, \mathsf{b})>\tilde{\mu}(H, s, \mathsf{w})\), then
\[
    \tilde{\mu}(H)<2\tilde{\mu}(H, s, \mathsf{b}).
\]
Let \(H^*\) be the graph obtained from \(H\) by contracting
\(V(R)\cup\{s\}\) into a single vertex. By
Remark~\ref{rmk: contraction_recoloring}, the valid 3-colorings counted by
\(\tilde{\mu}(H, s, \mathsf{b})\) induce valid 3-colorings of \(H^*\). Hence
\[
    \tilde{\mu}(H, s, \mathsf{b})\leq \tilde{\mu}(H^*).
\]
Adding a new disjoint triangle \(R'\), and setting
\[
    G'=R' + H^* + (G-V(H)),
\]
we obtain a graph \(G'\) on \(n\) vertices in \(\mathcal C\) such that
\[
    \tilde{\mu}(G')
    =4\tilde{\mu}(H^*)\tilde{\mu}(G-V(H))
    \geq 4\tilde{\mu}(H, s, \mathsf{b})\tilde{\mu}(G-V(H))
    >\tilde{\mu}(H)\tilde{\mu}(G-V(H))
    =\tilde{\mu}(G),
\]
a contradiction.

Assume now that
\[
    \tilde{\mu}(H, s, \mathsf{b})\leq \tilde{\mu}(H, s, \mathsf{w}).
\]
Then
\[
    \tilde{\mu}(H)\leq 2\tilde{\mu}(H, s, \mathsf{w}).
\]
If a valid 3-coloring of \(H\) is counted by
\(\tilde{\mu}(H, s, \mathsf{w})\), then its restriction to \(H-V(R)\) is a
valid 3-coloring of \(H-V(R)\). Indeed, \(s\) and \(v_3\) are white, and
\(v_1\) and \(v_2\) are yellow; since \(v_1v_2\in E(H)\), every neighbor
outside \(R\) of \(v_1\) or \(v_2\) must be white. Thus deleting \(R\) does
not destroy the validity of the coloring on the remaining vertices. Moreover,
this restriction is injective on the colorings counted by
\(\tilde{\mu}(H, s, \mathsf{w})\). Therefore
\[
    \tilde{\mu}(H, s, \mathsf{w})\leq \tilde{\mu}(H-V(R)).
\]
Let \(R'\) be a new triangle, and define
\[
    G'=R' + (H-V(R)) + (G-V(H)).
\]
Then \(G'\) is a graph on \(n\) vertices in \(\mathcal C\), and
\[
    \tilde{\mu}(G')
    =4\tilde{\mu}(H-V(R))\tilde{\mu}(G-V(H))
    \geq 4\tilde{\mu}(H, s, \mathsf{w})\tilde{\mu}(G-V(H))
    \geq 2\tilde{\mu}(H)\tilde{\mu}(G-V(H))
    >\tilde{\mu}(G),
\]
again a contradiction.

\medskip

\noindent\textbf{Case 3: \(d_R(s)=1\).}
Without loss of generality, assume that
\[
    N_H(s)\cap V(R)=\{v_1\}.
\]
Then \(V(R)\cup\{s\}\) induces a paw in \(H\). Since \(s\) was chosen
maximizing \(d_R(s)\), every vertex of \(V(H)\setminus V(R)\) has at most
one neighbor in \(R\). We analyze two cases: \(d_H(s)=1\), and
\(d_H(s)\geq 2\).

First suppose that \(d_H(s)=1\). Then every valid 3-coloring of
\(H[V(R)\cup\{s\}]\) falls into exactly one of the following three types:
\begin{enumerate}[(i)]
    \item \(v_1,v_2,v_3\) are colored black, and \(s\) is colored yellow;
    \item \(v_1,v_2\) are colored yellow, and \(v_3,s\) are colored white;
    \item \(v_1,v_3\) are colored yellow, and \(v_2,s\) are colored white.
\end{enumerate}
Let \(\alpha_1,\alpha_2,\alpha_3\) denote the numbers of valid 3-colorings of
\(H\) of these three types, respectively, and set
\[
    \alpha_{\max}:=\max\{\alpha_2,\alpha_3\}.
\]
Thus
\[
    \tilde{\mu}(H)=\alpha_1+\alpha_2+\alpha_3
    \leq \alpha_1+2\alpha_{\max}.
\]

If \(\alpha_1\geq \alpha_{\max}\), then
\[
    \tilde{\mu}(H)\leq \alpha_1+2\alpha_{\max}
    \leq 3\alpha_1<4\alpha_1.
\]
Let \(H^*\) be the graph obtained from \(H\) by contracting
\(V(R)\cup\{s\}\) into a single vertex. The valid 3-colorings counted by
\(\alpha_1\) induce valid 3-colorings of \(H^*\). Indeed, the only vertex of
\(V(R)\cup\{s\}\) that is not black is the leaf \(s\), and \(s\) has no
neighbor outside \(R\) because \(d_H(s)=1\). Thus the recoloring convention
of Remark~\ref{rmk: contraction_recoloring} applies. Consequently,
\[
    \alpha_1\leq \tilde{\mu}(H^*).
\]
Adding a new disjoint triangle \(R'\), we obtain
\[
    G'=R' + H^* + (G-V(H)).
\]
Then \(G'\) is a graph on \(n\) vertices in \(\mathcal C\), and
\[
    \tilde{\mu}(G')
    =4\tilde{\mu}(H^*)\tilde{\mu}(G-V(H))
    \geq 4\alpha_1\tilde{\mu}(G-V(H))
    >\tilde{\mu}(H)\tilde{\mu}(G-V(H))
    =\tilde{\mu}(G),
\]
a contradiction.

Assume now that \(\alpha_{\max}>\alpha_1\). Without loss of generality,
suppose first that \(\alpha_{\max}=\alpha_2\). The restriction to
\(H-V(R)\) of every valid 3-coloring counted by \(\alpha_2\) is a valid
3-coloring of \(H-V(R)\), and this restriction is injective on the colorings
counted by \(\alpha_2\). Hence
\[
    \alpha_{\max}=\alpha_2\leq \tilde{\mu}(H-V(R)).
\]
The case \(\alpha_{\max}=\alpha_3\) is symmetric. Letting
\[
    G'=R' + (H-V(R)) + (G-V(H)),
\]
where \(R'\) is a new triangle, we obtain a graph \(G'\) on \(n\) vertices in
\(\mathcal C\), and
\[
    \tilde{\mu}(G')
    =4\tilde{\mu}(H-V(R))\tilde{\mu}(G-V(H))
    \geq 4\alpha_{\max}\tilde{\mu}(G-V(H)).
\]
Since \(\alpha_{\max}>\alpha_1\), we have
\[
    \tilde{\mu}(H)=\alpha_1+\alpha_2+\alpha_3
    \leq \alpha_1+2\alpha_{\max}
    <3\alpha_{\max}<4\alpha_{\max},
\]
and therefore \(\tilde{\mu}(G')>\tilde{\mu}(G)\), again a contradiction.

It remains to consider the case \(d_H(s)\geq 2\). In this case every valid
3-coloring of \(H[V(R)\cup\{s\}]\) falls into exactly one of the following
five types:
\begin{enumerate}[(a)]
    \item \(v_1,v_2,v_3\) and \(s\) are colored black;
    \item \(v_1,v_2\) are colored yellow, and \(v_3,s\) are colored white;
    \item \(v_1,v_3\) are colored yellow, and \(v_2,s\) are colored white;
    \item \(v_1\) is colored white, while \(v_2,v_3\), and \(s\) are
    colored yellow;
    \item \(v_1,v_2,v_3\) are colored black, \(s\) is colored yellow, and
    all vertices in \(N_H(s)\setminus V(R)\) are colored white.
\end{enumerate}
In type (d), since \(s\) is yellow and \(v_1\) is white, exactly one vertex
of \(N_H(s)\setminus V(R)\) is colored non-white, and all remaining vertices
of \(N_H(s)\setminus V(R)\) are colored white.

Let \(\beta_1,\beta_2,\beta_3,\beta_4,\beta_5\) denote the numbers of valid
3-colorings of \(H\) of the corresponding types, and set
\[
    \beta_{\max}:=\max\{\beta_2,\beta_3\}.
\]
Then
\[
    \tilde{\mu}(H)
    =\beta_1+\beta_2+\beta_3+\beta_4+\beta_5
    \leq \beta_1+\beta_5+2\beta_{\max}+\beta_4.
\]

Suppose first that
\[
    2(\beta_{\max}+\beta_4)\geq \beta_1+\beta_5.
\]
Consider the graph \(H-V(R)\). The restrictions to \(H-V(R)\) of the valid
3-colorings counted by the larger of \(\beta_2\) and \(\beta_3\), together
with those counted by \(\beta_4\), are valid 3-colorings of \(H-V(R)\).
Moreover, these restrictions are injective within each family, and the two
families are disjoint because \(s\) is colored white in types (b) and (c),
whereas \(s\) is colored yellow in type (d). Hence
\[
    \beta_{\max}+\beta_4\leq \tilde{\mu}(H-V(R)).
\]
Let \(R'\) be a new triangle and define
\[
    G'=R' + (H-V(R)) + (G-V(H)).
\]
Then \(G'\) is a graph on \(n\) vertices in \(\mathcal C\). If
\(\beta_4>0\), then
\[
    \tilde{\mu}(H)
    \leq \beta_1+\beta_5+2\beta_{\max}+\beta_4
    \leq 4\beta_{\max}+3\beta_4
    <4(\beta_{\max}+\beta_4).
\]
Therefore,
\[
    \tilde{\mu}(G')
    =4\tilde{\mu}(H-V(R))\tilde{\mu}(G-V(H))
    \geq 4(\beta_{\max}+\beta_4)\tilde{\mu}(G-V(H))
    >\tilde{\mu}(G).
\]
If \(\beta_4=0\), then the valid 3-coloring of \(H-V(R)\) associated with
the trivial PED-set is not among the restrictions counted by
\(\beta_{\max}\), because in those restrictions the vertex \(s\) is colored
white, while \(s\) has positive degree in \(H-V(R)\). Hence
\[
    \tilde{\mu}(H-V(R))\geq \beta_{\max}+1,
\]
and
\[
    \tilde{\mu}(G')
    =4\tilde{\mu}(H-V(R))\tilde{\mu}(G-V(H))
    \geq 4(\beta_{\max}+1)\tilde{\mu}(G-V(H)).
\]
Since \(2\beta_{\max}\geq \beta_1+\beta_5\) and
\(\beta_2+\beta_3\leq 2\beta_{\max}\), we have
\[
    \tilde{\mu}(H)
    =\beta_1+\beta_2+\beta_3+\beta_5
    \leq \beta_1+\beta_5+2\beta_{\max}
    \leq 4\beta_{\max}
    <4(\beta_{\max}+1).
\]
Thus \(\tilde{\mu}(G')>\tilde{\mu}(G)\) in this case as well, a
contradiction.

We may therefore assume that
\[
    \beta_1+\beta_5>2(\beta_{\max}+\beta_4).
\]
If \(\beta_1\geq \beta_5\), then, by
Remark~\ref{rmk: contraction_recoloring}, the valid 3-colorings counted by
\(\beta_1\) induce valid 3-colorings of the graph \(H^*\) obtained from
\(H\) by contracting \(V(R)\cup\{s\}\) into a single vertex. Hence
\[
    \beta_1\leq \tilde{\mu}(H^*).
\]
Adding a new disjoint triangle \(R'\), we obtain
\[
    G'=R' + H^* + (G-V(H)).
\]
This is a graph on \(n\) vertices in \(\mathcal C\). Since
\(\beta_1\geq \beta_5\) and
\(\beta_1+\beta_5>2(\beta_{\max}+\beta_4)\), we have
\[
    4\beta_1
    \geq 2(\beta_1+\beta_5)
    > \beta_1+\beta_5+2(\beta_{\max}+\beta_4)
    \geq \beta_1+\beta_5+\beta_2+\beta_3+\beta_4
    =\tilde{\mu}(H).
\]
Therefore,
\[
    \tilde{\mu}(G')
    =4\tilde{\mu}(H^*)\tilde{\mu}(G-V(H))
    \geq 4\beta_1\tilde{\mu}(G-V(H))
    >\tilde{\mu}(G),
\]
a contradiction.

Finally, assume that \(\beta_5>\beta_1\). Since
\(\beta_1+\beta_5>2(\beta_{\max}+\beta_4)\), we have
\[
    \beta_{\max}+\beta_4<\beta_5.
\]
Thus
\[
    \tilde{\mu}(H)
    =\beta_1+\beta_2+\beta_3+\beta_4+\beta_5
    \leq \beta_1+2\beta_{\max}+\beta_4+\beta_5
    <4\beta_5.
\]
Let
\[
    k=|N_H(s)\setminus V(R)|.
\]
Since \(d_H(s)\geq 2\) and \(N_H(s)\cap V(R)=\{v_1\}\), we have \(k\geq 1\).
Contract \(V(R)\) into a single vertex \(z\), and then delete the vertices
\[
    \{s\}\cup (N_H(s)\setminus V(R)).
\]
Let \(H_0\) be the resulting graph. This operation reduces the number of
vertices by \(k+3\): two vertices are lost when \(V(R)\) is contracted, and
then \(s\), together with the \(k\) vertices of \(N_H(s)\setminus V(R)\), is
deleted. Since \(k+3\geq 4\), let \(F\) be the disjoint union of one triangle
and, if necessary, further disjoint triangles and at most one \(K_1\) or one
\(K_2\), so that
\[
    |V(F)|=k+3.
\]
Then \(\tilde{\mu}(F)\geq 4\).

Every valid 3-coloring counted by \(\beta_5\) induces a valid 3-coloring of
\(H_0\). Indeed, in type (e), all vertices of \(V(R)\) are colored black, the
vertex \(s\) is colored yellow, and all vertices in
\(N_H(s)\setminus V(R)\) are colored white. After contracting \(V(R)\), we
delete \(s\) and the white vertices in \(N_H(s)\setminus V(R)\). This
deletion does not affect the validity of the coloring on the remaining
vertices, except possibly at the contracted vertex \(z\); if necessary, \(z\)
is recolored according to its degree in \(H_0\). Thus the resulting coloring
is valid. Moreover, the map is injective on the colorings counted by
\(\beta_5\). Hence
\[
    \beta_5\leq \tilde{\mu}(H_0).
\]
Now define
\[
    G'=F+H_0+(G-V(H)).
\]
Then \(G'\) is a graph on \(n\) vertices in \(\mathcal C\), and
\[
    \tilde{\mu}(G')
    =\tilde{\mu}(F)\tilde{\mu}(H_0)\tilde{\mu}(G-V(H))
    \geq 4\beta_5\tilde{\mu}(G-V(H))
    >\tilde{\mu}(H)\tilde{\mu}(G-V(H))
    =\tilde{\mu}(G),
\]
which is the final contradiction.

All cases lead to a contradiction. Hence \(H=R\), as desired.
\end{proof}

\begin{rmk}
Theorem~\ref{thm: triangle} applies in particular when extremality is taken among all graphs on \(n\) vertices. In the sequel, we use the chordal version. The only additional point needed in the chordal setting is that each transformed graph remains chordal, which follows from Lemma~\ref{lem: chordal_operations}.
\end{rmk}

\begin{cor} \label{cor: decomposition chordal}
Let \(G\) be an extremal chordal graph on \(n\) vertices. Then \(G = G_1 + G_2\), where \(G_1\) is a disjoint union of triangles and \(G_2\) is a forest, with either \(G_1\) or \(G_2\) possibly empty.
\end{cor}

The preceding corollary reduces the extremal problem for chordal graphs to
the following situation. Every triangular component of an extremal chordal
graph is an isolated \(K_3\), which contributes a multiplicative factor
\(\tilde{\mu}(K_3)=4\). The remaining vertices induce a chordal
triangle-free graph, and hence a forest. Thus the extremal chordal graphs are
obtained by balancing isolated triangles with the extremal forests determined
in the previous subsection.

\begin{theorem} \label{thm: chordal_extremal}
If \(G\) is a chordal graph on \(n\) vertices, then \(\tilde{\mu}(G)\leq f(n)\), where
\[
f(n) = \begin{cases}
    1 & \text{if } n\leq 2, \\
    4^{\frac{n}{3}} & \text{if } n\geq 3 \text{ and } n\equiv 0 \pmod 3, \\
    4^{\frac{n-1}{3}} & \text{if } n\geq 4 \text{ and } n\equiv 1 \pmod 3, \\
    5\cdot 4^{\frac{n-5}{3}} & \text{if } n\geq 5 \text{ and } n\equiv 2 \pmod 3.
\end{cases}
\]
Furthermore, \(G\) is a chordal graph on \(n\) vertices such that \(\tilde{\mu}(G)=f(n)\) if and only if \(G\in \mathcal G\), where
\begin{eqnarray*}
\mathcal G &=& \{K_1, K_2, 2K_1\} \cup
\left\{ \frac{n}{3}K_3:
n\geq 3 \text{ and } n\equiv 0 \pmod 3 \right\} \cup \\
&& \left\{
\frac{n-4}{3}K_3 + P_4,\,
\frac{n-4}{3}K_3 + K_{1,3},\,
\frac{n-1}{3}K_3 + K_1:
n\geq 4 \text{ and } n\equiv 1 \pmod 3
\right\} \cup \\
&& \left\{
\frac{n-5}{3}K_3 + P_5,\,
\frac{n-5}{3}K_3 + K_{1,4}:
n\geq 5 \text{ and } n\equiv 2 \pmod 3
\right\}.
\end{eqnarray*}
\end{theorem}

\begin{proof}
Let \(G_0\) be an extremal chordal graph on \(n\) vertices. It suffices to prove that \(\tilde{\mu}(G_0)\leq f(n)\). We proceed by induction on \(n\).

The cases \(1\leq n\leq 5\) are verified directly using Corollary~\ref{cor: decomposition chordal}. Assume that \(n>5\), and that each extremal chordal graph on \(k<n\) vertices has at most \(f(k)\) perfect edge dominating sets.

If \(G_0\) contains a triangle \(R\), then by Theorem~\ref{thm: triangle}, applied to the class of chordal graphs, this triangle is a connected component of \(G_0\). Therefore,
\[
    G_0=R+H_0,
\]
where \(H_0\) is a chordal graph on \(n-3\) vertices. Since \(G_0\) is extremal, \(H_0\) must be an extremal chordal graph on \(n-3\) vertices. Indeed, otherwise replacing \(H_0\) by a chordal graph on \(n-3\) vertices with more PED-sets would increase \(\tilde{\mu}(G_0)\). By the inductive hypothesis,
\[
    \tilde{\mu}(H_0)\leq f(n-3),
\]
and hence
\[
    \tilde{\mu}(G_0)=4\tilde{\mu}(H_0)\leq 4f(n-3)=f(n).
\]

On the other hand, if \(G_0\) contains no triangle, then \(G_0\) is a triangle-free chordal graph, and therefore a forest. For \(6\leq n\leq 12\), the values listed in Remark~\ref{rmk: forests on six-twelve vertices} are strictly smaller than \(f(n)\). For \(n\geq 13\), by Theorem~\ref{thm: maxnumber pedforest}, the maximum among forests on \(n\) vertices is attained by \(P_n\). We now show that \(\tilde{\mu}(P_n)<f(n)\) for every \(n\geq 13\).

By Theorem~\ref{thm: recursive formula}, we have
\[
    \tilde{\mu}(P_{13})=116,\qquad
    \tilde{\mu}(P_{14})=170,\qquad
    \tilde{\mu}(P_{15})=249,
\]
and these values are less than \(f(13)\), \(f(14)\), and \(f(15)\), respectively. Suppose now that \(n>15\), and assume inductively that \(\tilde{\mu}(P_i)<f(i)\) for every \(13\leq i<n\). Using Theorem~\ref{thm: recursive formula}, we obtain
\[
    \tilde{\mu}(P_n)
    =\tilde{\mu}(P_{n-1})+\tilde{\mu}(P_{n-3})
    < f(n-1)+f(n-3).
\]
A direct verification in the three residue classes modulo \(3\) gives
\[
    f(n-1)+f(n-3)<f(n).
\]
For example, if \(n\equiv 0\pmod 3\), then $n-1\equiv 2\pmod 3$ and $n-3\equiv 0\pmod 3$, thus
\[
    f(n-1)=5\cdot 4^{\frac{n}{3}-2}
    \quad\text{and}\quad
    f(n-3)=4^{\frac{n}{3}-1},
\]
and therefore
\[
    f(n-1)+f(n-3)
    =9\cdot 4^{\frac{n}{3}-2}
    <16\cdot 4^{\frac{n}{3}-2}
    =4^{\frac{n}{3}}
    =f(n).
\]
The remaining two cases are analogous. Hence
\[
    \tilde{\mu}(P_n)<f(n)
\]
for all \(n\geq 13\). Therefore, when \(G_0\) is triangle-free and \(n\geq 6\), we have \(\tilde{\mu}(G_0)<f(n)\).

Consequently, in both cases we obtain
\[
    \tilde{\mu}(G_0)\leq f(n).
\]
This proves the upper bound.

We now prove the characterization of equality. Let \(G\) be a chordal graph on \(n\) vertices such that \(\tilde{\mu}(G)=f(n)\). We prove, by induction on \(n\), that \(G\in \mathcal G\).

For \(n=1\), \(G\) is isomorphic to \(K_1\). For \(n=2\), it is isomorphic to \(K_2\) or \(2K_1\). For \(n=3\), equality is attained by \(K_3\). For \(n=4\), equality is attained by \(P_4\), \(K_{1,3}\), and \(K_1+K_3\). For \(n=5\), equality is attained by \(P_5\) and \(K_{1,4}\). Hence \(G\in \mathcal G\) for \(n\leq 5\).

Suppose that \(n>5\), and assume that every chordal graph on \(l<n\) vertices with exactly \(f(l)\) PED-sets belongs to \(\mathcal G\).

We claim that \(G\) must contain a triangle. Otherwise, \(G\) would be a forest on \(n\geq 6\) vertices, and the preceding argument shows that \(\tilde{\mu}(G)<f(n)\), a contradiction. Therefore, \(G\) contains a triangle \(R\). By Theorem~\ref{thm: triangle}, applied to the class of chordal graphs, \(R\) is a connected component of \(G\). Hence
\[
    G=R+(G-V(R)).
\]
Moreover,
\[
    \tilde{\mu}(G)=f(n)
    \quad\text{and}\quad
    \tilde{\mu}(R)=4
\]
imply
\[
    \tilde{\mu}(G-V(R))=f(n-3).
\]
By the inductive hypothesis, \(G-V(R)\in \mathcal G\), and consequently \(G\in\mathcal G\).

Conversely, if \(G\in\mathcal G\), then \(G\) is chordal and a direct multiplicative computation over its connected components gives \(\tilde{\mu}(G)=f(n)\).
\end{proof}

\section{Algorithms}~\label{sec: alg}

\subsection{Counting perfect edge dominating sets}~\label{subsec: ped-counting}

In this subsection, we present a linear time algorithm that returns the number of PED-sets of a connected chordal graph. For this purpose, we modify slightly the algorithms PEDP-GSP and PEDP-C, both presented in \cite{Lu-Ko-Tang-2002}.

A graph is called a \emph{generalized series-parallel graph} (GSP graph), denoted by $G(u, v)$, where the vertices $u$ and $v$ of the graph are called the \emph{left} and \emph{right terminals}, respectively, if it is connected and is obtained according to the following recursive definition:

\begin{enumerate}
     \item The complete graph $K_2$ is a GSP graph and is called the \emph{basis graph} for the class of generalized series-parallel graphs.
     \item Given two GSP graphs $G_1(u_1, v_1)$ and $G_2(u_2, v_2)$, the graph $G^*$ obtained by applying one of the following three operations to $G_1$ and $G_2$ is also a GSP graph.
     \begin{enumerate}
          \item \emph{Series-1 composition} ($S1$): Identify vertices $v_1$ and $u_2$ to obtain $G^*(u_1, v_2)$.
          \item \emph{Series-2 composition} ($S2$): Identify vertices $v_1$ and $u_2$ to obtain $G^*(u_1, v_1)$.
          \item \emph{Parallel composition} ($P$): Identify vertices $u_1$ and $u_2$ and also identify vertices $v_1$ and $v_2$ to obtain $G^*(u_1, v_1)$. It is assumed that no multiple edges will be created by this composition.
     \end{enumerate}
     \item Only graphs constructed by a finite number of applications of series-1, series-2 and parallel compositions are GSP graphs.
\end{enumerate}

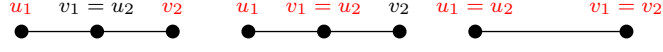
\begin{figure}
     \centering
     \begin{tikzpicture}[b/.style={circle, scale=0.5, draw=black, fill}, line width=0.5pt]
          \node (0) [b] at (0,0){};
          \node[above] at (0,0.1) {\footnotesize $\textcolor{red}{u_1}$};
          \node (1) [b] at (1,0){};
          \node[above] at (1,0.1) {\footnotesize $v_1=u_2$};
          \node (2) [b] at (2,0){};
          \node[above] at (2,0.1) {\footnotesize $\textcolor{red}{v_2}$};
          \node (a)[scale=0.9] at (1,1.6){(a) Series-1};
          \draw (0) -- (1) -- (2);
          
          \node (0) [b] at (3,0){};
          \node[above] at (3,0.1) {\footnotesize $\textcolor{red}{u_1}$};
          \node (1) [b] at (4,0){};
          \node[above] at (4,0.1) {\footnotesize $\textcolor{red}{v_1=u_2}$};
          \node (2) [b] at (5,0){};
          \node[above] at (5,0.1) {\footnotesize $v_2$};
          \node (a)[scale=0.9] at (4,1.6){(b) Series-2};
          \draw (0) -- (1) -- (2);
          
          \node (0) [b] at (6,0){};
          \node[above] at (6,0.1) {\footnotesize $\textcolor{red}{u_1=u_2}$};
          \node (1) [b] at (8,0){};
          \node[above] at (8,0.1) {\footnotesize $\textcolor{red}{v_1=v_2}$};
          \node (a)[scale=0.9] at (7,1.6){(c) Parallel};
          \draw (0) -- (1);
     \end{tikzpicture}
     \caption{Graphs obtained by applying the three compositions to the basis graph $K_2$. The terminals are highlighted in red.}
     \label{Fig: pictures of compositions}
\end{figure}

Let $G$ be any graph, let $x$ and $z$ be two vertices of $G$, and let $c$ and $c'$ be two colors, which can be black, yellow, or white. We define the following three functions:

\begin{itemize}
     \item Let $\Upsilon(G, x, z, c, c')$ denote the number of PED-sets of $G$ in which $x$ receives color $c$ and $z$ receives color $c'$.
     \item Let $\Phi(G, x, z, c)$ denote the number of colorings of $G$ in which $x$ is colored yellow, all its neighbors in $V(G)$ are colored white, $z$ receives color $c$, and the remaining vertices are colored so as  to form a valid 3-coloring of $G-x$.
     \item Let $\Gamma(G, x, z)$ denote the number of colorings of $G$ in which both $x$ and $z$ are colored yellow, all their neighbors in $V(G)$ are colored white, and the remaining vertices are colored so as to form a valid 3-coloring of $G-\{x, z\}$.
\end{itemize}

Note that if $x$ and $z$ are adjacent, then $\Gamma(G, x, z)=0$.

Let $G=G(u, v)$ be a generalized series-parallel graph, and let $\sigma$ be a valid 3-coloring of $G$. For $\alpha, \beta\in\{0, 1, 2, 3, 4\}$, we define:

\[
p_{\alpha, \beta}(G) = \begin{cases}
     \Upsilon(G, u, v, \sigma(u), \sigma(v)) & \text{if } \alpha \neq 4 \,\, \text{and } \beta \neq 4, \\
     \Phi(G, u, v, \sigma(v)) & \text{if } \alpha = 4 \,\, \text{and } \beta \neq 4, \\
     \Phi(G, v, u, \sigma(u)) & \text{if } \alpha \neq 4 \,\, \text{and } \beta = 4, \\
     \Gamma(G, u, v) & \text{if } \alpha = 4 \,\, \text{and } \beta = 4,
\end{cases}
\]

where
\begin{itemize}
     \item $\sigma(u)=\mathsf{w}$ (resp. $\sigma(v)=\mathsf{w}$) if $\alpha=0$ (resp. $\beta=0$).
     \item $\sigma(u)=\mathsf{y}$ (resp. $\sigma(v)=\mathsf{y}$) and $u$ (resp. $v$) has degree at least two in $G$ if $\alpha=1$ (resp. $\beta=1$).
     \item $\sigma(u)=\mathsf{y}$ (resp. $\sigma(v)=\mathsf{y}$) and $u$ (resp. $v$) has degree one in $G$ if $\alpha=2$ (resp. $\beta=2$).
     \item $\sigma(u)=\mathsf{b}$ (resp. $\sigma(v)=\mathsf{b}$) if $\alpha=3$ (resp. $\beta=3$).
\end{itemize}

If $K$ is a graph isomorphic to $K_2$, it is easy to verify that $p_{4, 0}(K)=p_{0, 4}(K)=p_{2, 2}(K)=1$, and that for any other combinations of values of $\alpha$ and $\beta$, we have $p_{\alpha, \beta}(K)=0$. Furthermore, the number of PED-sets of a GSP graph $G$ is $\tilde{\mu}(G)=\displaystyle \sum_{\alpha, \beta\in I} p_{\alpha, \beta}(G)$, where $I:=\{0, 1, 2, 3\}$. Now, we need to know the values of $p_{\alpha, \beta}(G)$ according to how $G$ is obtained from the GSP graphs $G_1$ and $G_2$. For this purpose, we consider the following three lemmas, where we only prove the first of them, since the proof of the remaining two is similar.

\begin{lem}~\label{lem: numped series-1}
Let $G=G(u, v)$ be a GSP graph obtained from the GSP graphs $G_1=G_1(u_1, v_1)$ and $G_2=G_2(u_2, v_2)$ by a series-1 composition. Then for all $\alpha, \beta \in I\cup\{4\}$,
$$p_{\alpha, \beta}(G) = p_{\alpha, 0}(G_1) \, p_{0, \beta}(G_2) + (p_{\alpha, 1}(G_1) + p_{\alpha, 2}(G_1)) \, p_{4, \beta}(G_2) + p_{\alpha, 4}(G_1) \, (p_{1, \beta}(G_2) + p_{2, \beta}(G_2))$$
$$+ (p_{\alpha, 2}(G_1) + p_{\alpha, 3}(G_1)) \, (p_{2, \beta}(G_2) + p_{3, \beta}(G_2)).$$
\end{lem}

\begin{proof}
Suppose that $G$ is obtained from $G_1$ and $G_2$ by applying a series-1 composition. In this operation, the vertices $v_1$ and $u_2$ are identified into a single vertex $a$, and the terminals of $G$ are $u_1$ and $v_2$, with labels $\alpha$ and $\beta$, respectively. If $a$ is colored white, then the number of 3-colorings of $G$ is $p_{\alpha, 0}(G_1) p_{0, \beta}(G_2)$. If $a$ is colored yellow, then its unique neighbor colored non-white lies in exactly one of $G_1$ and $G_2$. If this neighbor lies in $G_1$, then all neighbors of $a$ in $G_2$ are colored white. Thus, the number of 3-colorings of $G$ in this case is $(p_{\alpha, 1}(G_1) + p_{\alpha, 2}(G_1)) \, p_{4, \beta}(G_2)$. Likewise, if this neighbor lies in $G_2$, then the number of 3-colorings of $G$ is $p_{\alpha, 4}(G_1) \, (p_{1, \beta}(G_2) + p_{2, \beta}(G_2))$. Therefore, the total number of 3-colorings of $G$ when $a$ is colored yellow is $(p_{\alpha, 1}(G_1) + p_{\alpha, 2}(G_1)) \, p_{4, \beta}(G_2) + p_{\alpha, 4}(G_1) \, (p_{1, \beta}(G_2) + p_{2, \beta}(G_2))$. Finally, suppose that $a$ is colored black. Then all neighbors of $a$ in both $G_1$ and $G_2$ must be colored non-white. Observe that if $a$ is colored yellow and satisfies $d_{G_1}(a)=d_{G_2}(a)=1$, then $d_G(a)=2$, and recoloring $a$ from yellow to black, the resulting 3-coloring of $G$ is valid. Therefore, the corresponding number of 3-colorings of $G$ when $a$ is colored black is $(p_{\alpha, 2}(G_1) + p_{\alpha, 3}(G_1)) \, (p_{2, \beta}(G_2) + p_{3, \beta}(G_2))$. This completes the analysis of all possible colors of $a$.
\end{proof}

\begin{lem}~\label{lem: numped series-2}
Let $G=G(u, v)$ be a GSP graph obtained from the GSP graphs $G_1=G_1(u_1, v_1)$ and $G_2=G_2(u_2, v_2)$ by a series-2 composition. Then for all $\alpha \in I\cup\{4\}$,
\begin{enumerate}
     \item $p_{\alpha, 0}(G)=p_{\alpha, 0}(G_1) \, \displaystyle \sum_{\beta \in I} p_{0, \beta}(G_2)$.
     \item $p_{\alpha, 1}(G)=(p_{\alpha, 1}(G_1) + p_{\alpha, 2}(G_1)) \, \displaystyle \sum_{\beta \in I} p_{4, \beta}(G_2) + p_{\alpha, 4}(G_1) \, \sum_{\beta \in I} (p_{1, \beta}(G_2) + p_{2, \beta}(G_2))$.
     \item $p_{\alpha, 2}(G)=0$.
     \item $p_{\alpha, 3}(G)=(p_{\alpha, 2}(G_1) + p_{\alpha, 3}(G_1)) \, \displaystyle \sum_{\beta \in I} (p_{2, \beta}(G_2) + p_{3, \beta}(G_2))$.
     \item $p_{\alpha, 4}(G)=p_{\alpha, 4}(G_1) \, \displaystyle \sum_{\beta \in I} p_{4, \beta}(G_2)$.
\end{enumerate}
\end{lem}

\begin{lem}~\label{lem: numped parallel}
Let $G=G(u, v)$ be a GSP graph obtained from the GSP graphs $G_1=G_1(u_1, v_1)$ and $G_2=G_2(u_2, v_2)$ by a parallel composition. Then,
\begin{enumerate}
     \item $p_{0, 0}(G)=p_{0, 0}(G_1) \, p_{0, 0}(G_2)$.
     \item $p_{0, 1}(G)=(p_{0, 1}(G_1) + p_{0, 2}(G_1)) \, p_{0, 4}(G_2) + p_{0, 4}(G_1) \, (p_{0, 1}(G_2) + p_{0, 2}(G_2))$.
     \item $p_{0, 3}(G)=(p_{0, 2}(G_1) + p_{0, 3}(G_1)) \, (p_{0, 2}(G_2) + p_{0, 3}(G_2))$.
     \item $p_{0, 4}(G)=p_{0, 4}(G_1) \, p_{0, 4}(G_2)$.
     \item $p_{1, 1}(G)=(p_{1, 1}(G_1) + p_{1, 2}(G_1) + p_{2, 1}(G_1) + p_{2, 2}(G_1)) \, p_{4, 4}(G_2) + p_{4, 4}(G_1) \, (p_{1, 1}(G_2) + p_{1, 2}(G_2) + p_{2, 1}(G_2) + p_{2, 2}(G_2)) + (p_{1, 4}(G_1) + p_{2, 4}(G_1)) \, (p_{4, 1}(G_2) + p_{4, 2}(G_2)) + (p_{4, 1}(G_1) + p_{4, 2}(G_1)) \, (p_{1, 4}(G_2) + p_{2, 4}(G_2))$.
     \item $p_{1, 3}(G)=(p_{1, 2}(G_1) + p_{1, 3}(G_1)) \, (p_{4, 2}(G_2) + p_{4, 3}(G_2)) + (p_{4, 2}(G_1) + p_{4, 3}(G_1)) \, (p_{1, 2}(G_2) + p_{1, 3}(G_2))$.
     \item $p_{1, 4}(G)=(p_{1, 4}(G_1) + p_{2, 4}(G_1)) \, p_{4, 4}(G_2) + p_{4, 4}(G_1) \, (p_{1, 4}(G_2) + p_{2, 4}(G_2))$.
     \item $p_{3, 3}(G)=(p_{2, 2}(G_1) + p_{2, 3}(G_1) + p_{3, 2}(G_1) + p_{3, 3}(G_1)) \, (p_{2, 2}(G_2) + p_{2, 3}(G_2) + p_{3, 2}(G_2) + p_{3, 3}(G_2))$.
     \item $p_{3, 4}(G)=(p_{2, 4}(G_1) + p_{3, 4}(G_1)) \, (p_{2, 4}(G_2) + p_{3, 4}(G_2))$.
     \item $p_{4, 4}(G)=p_{4, 4}(G_1) \, p_{4, 4}(G_2)$.
     \item $p_{\alpha, 2}(G)=p_{2, \beta}(G)=0$ for all $\alpha, \beta \in I\cup\{4\}$.
     \item Analogously, the values of $p_{\alpha, \beta}(G)$ with $\alpha>\beta$ are obtained from the items above by swapping the subscripts.
\end{enumerate}
\end{lem}

For any GSP graph $G$, its structure can be represented by a \emph{parse tree}, $PT(G)$, that is defined as a binary tree in which each node of $PT(G)$ represents a subgraph of $G$.
\begin{itemize}
     \item Each leaf of $PT(G)$ represents an edge of $G$. It should be noted that there is a one-to-one correspondence between the leaves of $PT(G)$ and the edges of $G$.
     \item Each internal node of $PT(G)$ has a label either $S1$, $S2$ and $P$. A node labeled as $S1$ (resp. $S2$ and $P$) represents the subgraph of $G$ obtained by applying a series-1 (resp. series-2 and parallel) composition to the subgraphs corresponding to its children.
     \item The root of $PT(G)$ represents $G$ itself.
\end{itemize}

Note that the parse tree of a generalized series–parallel graph may not be unique. Furthermore, there are linear time algorithms to recognize whether a graph $G$ is generalized series–parallel and to construct a parse tree $PT(G)$ of it \cite{KTY, Lu-Ko-Tang-2002, VTL-1982}.

We now focus on the framework of the linear-time algorithm PEDP-GSP from \cite{Lu-Ko-Tang-2002}, presenting the modifications introduced in Lemmas \ref{lem: numped series-1}, \ref{lem: numped series-2}, and \ref{lem: numped parallel} to count the number of perfect edge dominating sets of a generalized series-parallel graph. The main modification consists in considering the terminals colored yellow according to their degrees in $G_1$ and $G_2$. For example, if $v_1$ and $u_2$ have degree one in $G_1$ and $G_2$, respectively, and a series-1 composition is applied to $G_1$ and $G_2$, then the identified vertex $v_1=u_2$ in $G$ must be colored black, resulting in a valid 3-coloring of $G$. On the contrary, if one of these terminals colored yellow has degree at least two, then it has at least one neighbor colored white, and after applying a series-1 composition, the resulting 3-coloring of $G$ is not valid.

Let $G$ be a generalized series-parallel graph with $|V(G)|=n$ and $|E(G)|=m$. Once we construct a parse tree $PT(G)$, we do
\begin{enumerate}
     \item Process $PT(G)$ in a bottom-up fashion with the following initialization: for each leaf $K$ of $PT(G)$, do $p_{4, 0}(K) = p_{0, 4}(K) = p_{2, 2}(K) = 1$, and $p_{\alpha, \beta}(K) = 0$ otherwise. Mark this leaf.
     \item Choose some unmarked vertex $v_i$ of $PT(G)$ whose children are marked.
     \begin{enumerate}
          \item If vertex $v_i$ is labeled as $S1$, then compute all $p_{\alpha, \beta}(G_i)$ according to Lemma \ref{lem: numped series-1}, where $\alpha, \beta \in I \cup \{4\}$.
          \item If vertex $v_i$ is labeled as $S2$, then compute all $p_{\alpha, \beta}(G_i)$ according to Lemma \ref{lem: numped series-2}, where $\alpha, \beta \in I \cup \{4\}$.
          \item If vertex $v_i$ is labeled as $P$, then compute all $p_{\alpha, \beta}(G_i)$ according to Lemma \ref{lem: numped parallel}, where $\alpha, \beta \in I \cup \{4\}$.
     \end{enumerate}
     Mark the vertex $v_i$.
     \item Finally, output $\tilde{\mu}(G)=\displaystyle \sum_{\alpha, \beta \in I} p_{\alpha, \beta}(G)$.
\end{enumerate}

With these settings, we obtain the following result.

\begin{theorem}~\label{thm: ped counting gsp}
The PED-counting problem in generalized series-parallel graphs can be solved in $O(n+m)$ time.
\end{theorem}

With this preparation, we now turn to chordal graphs.

In \cite{Ga-1974}, Gavril proved that a graph $G$ is chordal if and only if it is the intersection graph of a family of subtrees in an undirected tree. Moreover, he proved that there exists a tree $T$ whose vertex set is the set of all maximal cliques of $G$, and each vertex $v$ in $G$ corresponds to the subtree of $T$ comprised of exactly those maximal cliques to which $v$ belongs. Such tree $T$ is called a \emph{clique tree} of $G$, and it is denoted by $CT(G)$. It is known that a clique tree of a chordal graph can be obtained in linear time \cite{Ga-1974}.

\begin{rmk}~\label{rmk: pedcomplete} \cite{Li-Lo-Mo-Sz}
Every complete graph on at least four vertices has all its vertices colored black.
\end{rmk}

Remark \ref{rmk: pedcomplete} implies that all vertices of every clique of $G$ on at least four vertices must be colored black. Therefore, the vertices of clique tree $CT(G)$ corresponding to such cliques can be deleted. Let $\mathcal{C}=\{C_1, C_2, \ldots, C_p\}$ denote these cliques, and let $V(\mathcal{C})=\cup_{1\leq i\leq p} C_i$. Deleting their corresponding vertices from $CT(G)$ disconnects the clique tree, resulting in a forest denoted by $\mathcal{F}=\{T_1, T_2, \ldots, T_q\}$. Clearly, each $T_i\in\mathcal{F}$ contains no clique of size greater than three. If $G_i$ is the subgraph of $G$ induced by the union of all cliques in $T_i$, then $T_i$ is a clique tree of $G_i$, and $G_i$ is a chordal graph with no clique of size greater than three. Moreover, each $G_i$ is a GSP graph due to the following lemma.

\begin{lem}~\label{lem: chorgsp} \cite{Lu-Ko-Tang-2002}
A chordal graph with no clique on at least four vertices is a generalized series-parallel graph.
\end{lem}

Note that Remark \ref{rmk: pedcomplete} implies that if $P$ is a PED-set of $G$, then $P\cap E(G_i)$ is a PED-set of $G_i$, and contains all edges of it incident with some vertices of $V(\mathcal{C})$. Let $\overline{\mathcal{P}}(G_i)$ be a PED-set of $G_i$ with the restriction that it contains all edges of $G_i$ incident with some vertices of $V(\mathcal{C})$.

\begin{lem} \cite{Lu-Ko-Tang-2002}
Let $\mathcal{P}=(\cup_{1\leq i\leq q} \overline{\mathcal{P}}(G_i))\cup (\cup_{1\leq j\leq p} \{uv: u,v\in C_j, u\neq v\})$. Then $\mathcal{P}$ is a PED-set of $G$.
\end{lem}

The number of PED-sets $\overline{\mathcal{P}}(G_i)$ that $G_i$ admits will be denoted by $\tilde{\mu}_\mathcal{C}(G_i)$ for each $1\leq i\leq q$. Let $G$ be a connected chordal graph with $|V(G)|=n$ and $|E(G)|=m$. Just like the linear-time algorithm PEDP-C from \cite{Lu-Ko-Tang-2002}, we find the set $\mathcal{C}$ of cliques on at least four vertices and remove them from $CT(G)$ to obtain a forest $\mathcal{F}=\{T_1, T_2, \ldots, T_q\}$. Let $G_j$ be a generalized series-parallel graph whose clique tree is $T_j$ for each $1\leq j\leq q$.
\begin{enumerate}
     \item For each $1\leq j\leq q$, compute $\tilde{\mu}_\mathcal{C}(G_j)$ by processing the corresponding parse tree as in the previous algorithm, using the following initialization: for each edge $uv\in E(G_j)$, do
     \begin{itemize}
          \item If $u, v \in V(\mathcal{C})$, then $p_{3, 3}(G[\{u, v\}])=1$, and the other cases are zero.
          \item If $u, v \not\in V(\mathcal{C})$, then $p_{4, 0}(G[\{u, v\}])=p_{0, 4}(G[\{u, v\}])=p_{2, 2}(G[\{u, v\}])=1$, and the other cases are zero.
           \item If $u \in V(\mathcal{C})$ and $v \not\in V(\mathcal{C})$, then $p_{3, 2}(G[\{u, v\}])=1$, and the other cases are zero.
           \item If $u \not\in V(\mathcal{C})$ and $v \in V(\mathcal{C})$, then $p_{2, 3}(G[\{u, v\}])=1$, and the other cases are zero.
     \end{itemize}
     \item Finally, output $\tilde{\mu}(G)=\displaystyle \prod_{j=1}^q \tilde{\mu}_\mathcal{C}(G_j)$.
\end{enumerate}

For non-connected chordal graphs, we apply this algorithm to all their connected components. Consequently, we have the following theorem.

\begin{theorem}~\label{thm: ped counting chordal}
The PED-counting problem in chordal graphs can be solved in $O(n+m)$ time.
\end{theorem}

\subsection{Counting dominating induced matchings}~\label{subsec: dim-counting}

The same dynamic-programming approach also yields linear-time algorithms for counting dominating induced matchings. Recall that a DIM is a PED-set whose associated valid 3-coloring uses only yellow and white vertices. Thus, in this setting, the color black is forbidden.

Let \(G=G(u,v)\) be a generalized series-parallel graph, and let \(\sigma\) be a valid 3-coloring of \(G\) such that
\(\sigma(x)\neq \mathsf{b}\) for every \(x\in V(G)\). For \(\alpha,\beta\in\{0,1,2\}\), we define
\[
d_{\alpha,\beta}(G) =
\begin{cases}
     \Upsilon(G,u,v,\sigma(u),\sigma(v))
     & \text{if } \alpha \neq 2 \text{ and } \beta \neq 2, \\
     \Phi(G,u,v,\sigma(v))
     & \text{if } \alpha = 2 \text{ and } \beta \neq 2, \\
     \Phi(G,v,u,\sigma(u))
     & \text{if } \alpha \neq 2 \text{ and } \beta = 2, \\
     \Gamma(G,u,v)
     & \text{if } \alpha = 2 \text{ and } \beta = 2,
\end{cases}
\]
where \(\sigma(u)=\mathsf{w}\) if \(\alpha=0\), \(\sigma(u)=\mathsf{y}\) if \(\alpha=1\), and analogously \(\sigma(v)=\mathsf{w}\) if \(\beta=0\), \(\sigma(v)=\mathsf{y}\) if \(\beta=1\). Observe that, in this setting, we do not need to distinguish whether a yellow terminal has degree one or at least two.

Let \(K\) be the basis graph isomorphic to \(K_2\). We initialize
\[
     d_{1,1}(K)=d_{2,0}(K)=d_{0,2}(K)=1,
\]
and all other values \(d_{\alpha,\beta}(K)\) are set to \(0\). With this initialization, the same dynamic-programming updates used in Lemmas~\ref{lem: numped series-1}, \ref{lem: numped series-2}, and \ref{lem: numped parallel}, restricted to the states \(\{0,1,2\}\), compute all values \(d_{\alpha,\beta}(G)\) in constant time per composition step of the parse tree.

Consequently, the number of DIMs of a generalized series-parallel graph is obtained by processing its parse tree in linear time.

\begin{theorem}~\label{thm: dim counting gsp}
The DIM-counting problem in generalized series-parallel graphs can be solved
in \(O(n+m)\) time.
\end{theorem}

We now consider chordal graphs. If a graph contains a clique on at least four vertices, then it has no DIM~\cite{Lu-Ko-Tang-2002}. Hence, given a chordal graph \(G\), we first use a depth-first search on a clique tree of \(G\) to check whether \(G\) contains a clique on at least four vertices. If such a clique exists, then the number of DIMs of \(G\) is \(0\). Otherwise, by Lemma~\ref{lem: chorgsp}, the graph \(G\) is a generalized series-parallel
graph. We can then apply the algorithm described above.

For non-connected chordal graphs, we apply the algorithm to each connected component and multiply the resulting numbers. We obtain the following result.

\begin{theorem}~\label{thm: dim counting chordal}
The DIM-counting problem in chordal graphs can be solved in \(O(n+m)\) time.
\end{theorem}

\section{Conclusions and future research}~\label{sec: conc}

In this paper, we derived a recursive formula for the number of perfect edge dominating sets of the path on \(n\) vertices. We also determined the extremal graphs in the classes of trees, forests, and chordal graphs, namely,
those graphs on \(n\) vertices that maximize the number of perfect edge dominating sets within each class. On the algorithmic side, based on the ideas of Lu, Ko, and Tang~\cite{Lu-Ko-Tang-2002}, we obtained linear-time
algorithms for PED-counting in generalized series-parallel graphs and chordal graphs. We also showed that the same dynamic-programming approach yields linear-time algorithms for DIM-counting in both graph classes.

Regarding future research, the following problem remains open.

\begin{conj}~\label{conj: pedcon}
If \(G\) is a connected graph on \(n\) vertices, then
\[
     \tilde{\mu}(G)\leq \tilde{\mu}(C_n).
\]
\end{conj}

At present, Conjecture~\ref{conj: pedcon} is known to hold for paths, and hence for trees. Indeed, one can easily verify that \(\tilde{\mu}(C_3)=4\), \(\tilde{\mu}(C_4)=5\), and \(\tilde{\mu}(C_5)=6\), and the following theorem shows that every tree satisfies the conjecture.

\begin{theorem}~\label{thm: recfor cycle}
If \(n\geq 6\), then
\[
     \tilde{\mu}(C_n)=\tilde{\mu}(C_{n-1})+\tilde{\mu}(C_{n-3}).
\]
\end{theorem}

\begin{proof}
Let \(C\) be a graph isomorphic to the cycle \(C_n\), with vertex set \(V(C)=\{v_1, v_2, \ldots, v_n\}\), and edge set
\(E(C)=\{v_1v_2, v_2v_3, \ldots, v_{n-1} v_n, v_n v_1\}\).

We count the number of PED-sets of \(C\) by considering the color assigned
to the vertex \(v_1\).

First, suppose that \(v_1\) is colored white. Then \(v_2\) and \(v_n\) must
be colored yellow, and the vertices \(v_3\) and \(v_{n-1}\) must receive a
non-white color. Removing the vertices \(v_n\), \(v_1\), and \(v_2\), and
adding the edge \(v_3v_{n-1}\), we obtain a cycle \(C'\) on \(n-3\) vertices.
Therefore,
\[
\tilde{\mu}(C,v_1,\mathsf{w})
 =
\tilde{\mu}(C',v_3,\mathsf{b})
+
\tilde{\mu}(C',v_4,\mathsf{w}),
\]
where \(\tilde{\mu}(C',v_4,\mathsf{w})\) counts the number of PED-sets of
\(C'\) in which \(v_3\) is colored yellow and its unique non-white neighbor
is \(v_{n-1}\).

Now, suppose that \(v_1\) is colored black. Then \(v_2\) and \(v_n\) must
receive non-white colors. Removing the vertex \(v_1\), and adding the edge
\(v_2v_n\), we obtain a cycle \(C''\) on \(n-1\) vertices. Therefore,
\[
\tilde{\mu}(C,v_1,\mathsf{b})
 =
\tilde{\mu}(C'',v_2,\mathsf{b})
+
\tilde{\mu}(C'',v_3,\mathsf{w}),
\]
where \(\tilde{\mu}(C'',v_3,\mathsf{w})\) counts the number of PED-sets of
\(C''\) in which \(v_2\) is colored yellow and its unique non-white neighbor
is \(v_n\).

Finally, assume that \(v_1\) is colored yellow. Then its unique non-white
neighbor is either \(v_2\) or \(v_n\). If \(v_2\) is colored non-white, then
\(v_n\) must be colored white; symmetrically, if \(v_n\) is colored
non-white, then \(v_2\) must be colored white.

\begin{enumerate}
     \item Suppose that \(v_2\) is colored white and \(v_n\) is colored
     non-white. Then \(v_3\) must be colored yellow, and \(v_4\) must be
     non-white. If \(v_n\) is colored yellow, then \(v_{n-1}\) must be
     colored white. Removing the vertices \(v_n\), \(v_1\), and \(v_2\), and
     adding the edge \(v_3v_{n-1}\), we obtain the cycle \(C'\) again. On
     the other hand, if \(v_n\) is colored black, then \(v_{n-1}\) must be
     non-white. Removing the vertex \(v_1\), changing the color of \(v_n\)
     from black to yellow, and adding the edge \(v_2v_n\), we obtain the
     cycle \(C''\) again. Consequently,
     \[
     \tilde{\mu}(C,v_2,\mathsf{w})
     =
     \tilde{\mu}(C',v_{n-1},\mathsf{w})
     +
     \tilde{\mu}(C'',v_2,\mathsf{w}),
     \]
     where \(\tilde{\mu}(C',v_{n-1},\mathsf{w})\) counts the number of
     PED-sets of \(C'\) in which \(v_3\) is colored yellow and its unique
     non-white neighbor is \(v_4\).

     \item Suppose that \(v_n\) is colored white and \(v_2\) is colored
     non-white. Then \(v_{n-1}\) must be colored yellow, and \(v_{n-2}\) must
     be non-white. If \(v_2\) is colored yellow, then \(v_3\) must be colored
     white. Removing the vertices \(v_{n-1}\), \(v_n\), and \(v_1\), and
     adding the edge \(v_2v_{n-2}\), we obtain the cycle \(C'\) again. On
     the other hand, if \(v_2\) is colored black, then \(v_3\) must be
     non-white. Removing the vertex \(v_1\), changing the color of \(v_2\)
     from black to yellow, and adding the edge \(v_2v_n\), we obtain the
     cycle \(C''\) again. Consequently,
     \[
     \tilde{\mu}(C,v_n,\mathsf{w})
     =
     \tilde{\mu}(C',v_3,\mathsf{w})
     +
     \tilde{\mu}(C'',v_n,\mathsf{w}),
     \]
     where \(\tilde{\mu}(C',v_3,\mathsf{w})\) counts the number of PED-sets
     of \(C'\) in which \(v_2\) is colored yellow and its unique non-white
     neighbor is \(v_3\).
\end{enumerate}

Combining the three cases for the color of \(v_1\), we obtain
\[
     \tilde{\mu}(C_n)=\tilde{\mu}(C_{n-1})+\tilde{\mu}(C_{n-3}).
\]
\end{proof}

Theorem~\ref{thm: recfor cycle} shows that \(\tilde{\mu}(C_n)\) satisfies the same recurrence as \(\tilde{\mu}(P_n)\), but with larger initial values. Hence,
\[
     \tilde{\mu}(P_n)<\tilde{\mu}(C_n)
\]
for all \(n\geq 3\).

We conclude with the following open problem.

\begin{conj}
If \(G\) is an extremal graph on \(n\geq 6\) vertices in the class of all graphs on \(n\) vertices, then $G\in \mathcal{G}$, where
\begin{eqnarray*}
\mathcal G &=&
\left\{ \frac{n}{3}C_3:
n\geq 6 \text{ and } n\equiv 0 \pmod 3 \right\} \cup \\
&& \left\{
\frac{n-4}{3}C_3 + C_4:
n\geq 7 \text{ and } n\equiv 1 \pmod 3 \right\} \cup \\
&& \left\{
\frac{n-8}{3}C_3 + 2C_4:
n\geq 8 \text{ and } n\equiv 2 \pmod 3 \right\}.
\end{eqnarray*}
\end{conj}

Theorem~\ref{thm: triangle} allows us to separate the triangles of \(G\), but so far it is not possible to do the same for induced cycles of length four.

\section{Acknowledgments}

We are grateful to Ezequiel Dratman for fruitful discussions on the enumeration of perfect edge dominating sets. In particular, these discussions helped us test the first Sage computations for paths and forests and led to the conjectural upper bound that is proved here as Theorem \ref{thm: maxnumber pedforest}.

%\bibliographystyle{plain}     
%%\bibliographystyle{abbrv}
%\bibliography{bio_camilo}

\end{document}